\pdfoutput=1
\documentclass[11pt]{amsart}
\usepackage[T1]{fontenc}
\usepackage[utf8]{inputenc}
\usepackage{amssymb,amsmath,amsthm,mathrsfs,array,graphicx,bbm,enumerate,wasysym,dsfont}
\usepackage[usenames, dvipsnames]{color}
\usepackage[all]{xy}
\usepackage{bbm}
\usepackage{fullpage}
\usepackage{float}
\usepackage{verbatim}
\usepackage{hyperref}

\newtheorem{thm}{Theorem}[section]
\newtheorem{lemma}[thm]{Lemma}
\newtheorem{rem}[thm]{Remark}
\newtheorem{defn}[thm]{Definition}
\newtheorem{prop}[thm]{Proposition}
\newtheorem{cor}[thm]{Corollary}

\numberwithin{equation}{section}

\newcommand{\Z}{\mathbb Z}
\newcommand{\R}{\mathbb R}
\newcommand{\C}{\mathbb C}

\newcommand{\N}{\mathbb N}
\newcommand{\D}{\mathbb D}
\newcommand{\E}{\mathbb E}

\renewcommand{\P}{\mathbb P}
\renewcommand{\1}{\mathbf 1}
\newcommand{\A}{\mathds A}
\newcommand{\ABP}[2]{\protect\rotatebox[origin=c]{180}{$#1\A$}}
\newcommand{\AB}{{\mathpalette\ABP\relax}}
\newcommand{\B}{\mathcal B}

\newcommand{\M}{\operatorname{M}}
\renewcommand{\L}{\mathcal L}

\renewcommand{\AA}{\mathcal A}
\renewcommand{\epsilon}{\varepsilon}

\newcommand{\eps}{\epsilon}

\definecolor{crimsonglory}{rgb}{0.75, 0.0, 0.2}

\renewcommand{\d}{{d}}

\newenvironment{proofMassive}{\paragraph{\textit{Proof of Lemma \ref{LemMassive}.}}}
{\hfill$\square$}

\begin{document}
\title{The first passage sets of the 2D Gaussian free field: convergence and isomorphisms.}

\author{Juhan Aru \and Titus Lupu \and Avelio Sep\'ulveda}

\address {
	Department of Mathematics,
	ETH Z\"urich,
	Rämistr. 101,
	8092 Z\"urich,
	Switzerland}
\email
{juhan.aru@math.ethz.ch}

\address{CNRS and LPSM, UMR 8001, Sorbonne Université, 4 place Jussieu, 75252 Paris cedex 05, France}
\email
{titus.lupu@upmc.fr}

\address{Univ Lyon, Université Claude Bernard Lyon 1, CNRS UMR 5208, Institut Camille Jordan, 69622 Villeurbanne, France}
\email
{sepulveda@math.lyon-1.fr}
\begin{abstract}
In a previous article, we introduced the first passage set (FPS) of constant level $-a$ of the two-dimensional continuum Gaussian free field (GFF) on finitely connected domains. Informally, it is the set of points in the domain that can be connected to the boundary by a path along which the GFF is greater than or equal to $-a$. This description can be taken as a definition of the FPS for the metric graph GFF, and it justifies the analogy with the first hitting time of $-a$ by a one-dimensional Brownian motion. In the current article, we prove that the metric graph FPS converges towards the continuum FPS in the Hausdorff metric. This allows us to show that the FPS of the continuum GFF can be represented as a union of clusters of Brownian excursions and Brownian loops, and to prove that Brownian loop soup clusters admit a non-trivial Minkowski content in the gauge $r\mapsto |\log r|^{1/2}r^2$. We also show that certain natural interfaces of the metric graph GFF converge to SLE$_4$ processes.
\end{abstract}

\subjclass[2010]{60G15; 60G60; 60J65; 60J67; 81T40} 
\keywords{conformal loop ensemble; Gaussian free field; isomorphism theorems; local set; loop-soup; metric graph; Schramm-Loewner evolution}

\maketitle

\section{Introduction}

In this article, we continue the study of the first passage sets (FPS) of the 2D continuum Gaussian free field (GFF), initiated in \cite{ALS1}. Here, we cover different aspects of it: the approximation by metric graphs and the construction as clusters of two-dimensional Brownian loops and excursions.

The continuum (massless) Gaussian free field, known as bosonic massless free field in Euclidean quantum field theory \cite{Simon1974EQFT,GawedzkiCFT}, is a canonical model of a Gaussian field satisfying a spatial Markov property. In dimension $d\geq 2$, it is a generalized function, not defined pointwise. In dimension $d=2$, it is conformally invariant in law.

A key notion in the study of the GFF is that of local sets
\cite{SchSh2,WWln2,Se}, along which the GFF admits a Markovian decomposition. For the 2D GFF important examples are level lines
\cite{SchSh2,She05,Dub,WaWu}, flow lines
\cite{MS1,MS2,MS3,MS4}, and two-valued local sets \cite{ASW,AS2}. These are examples of thin local sets, that is to say they are not "charged" by the GFF and only the induced boundary values matter for the Markovian decomposition.

In \cite{ALS1}, we introduced a family of different non-thin local sets: the first passage sets (FPS). Although the 2D continuum GFF is not defined pointwise, one can imagine an FPS of level $-a$ as all the points in $\overline D$ that can be reached  from $\partial D$ by a continuous path along which the GFF has values
$\geq -a$. In some sense, an FPS is analogous to the first passage time of a Brownian motion, analogy we develop in \cite{ALS1}. Although an FPS has a.s. zero Lebesgue measure, the restriction of the GFF to it is in general non-trivial. It is actually a positive measure, a Minkowski content measure in the gauge $r\mapsto |\log r|^{1/2}r^2$. In this case, the behavior of the GFF on this local set is entirely determined by the geometry of the set itself. Observe that this differs from the one-dimensional case of Brownian first passage bridges.

In this article, we  make the above heuristic description of the FPS exact by approximating the continuum GFF by metric graph GFF-s. A metric graph is obtained by taking a discrete electrical network and replacing each edge by a continuous line segment of length proportional to the resistance (inverse of the conductance) of the edge. On the metric graph, one can define a Gaussian free field by interpolating  discrete GFF on vertices by conditional independent Brownian bridges inside the edges \cite{Lupu2016Iso}. Such a field is pointwise defined, continuous, and still satisfies a domain Markov property, even when cutting the domain inside the edges. For a metric graph GFF, the first passage set of level $-a$ is exactly defined by the heuristic description given in the previous paragraph: it is the set of points on the metric graph that are joined to the boundary by some path, on which the metric graph GFF does not go below the level $-a$ \cite{LupuWerner2016Levy}.  

The main result of this paper is Proposition \ref{Convergence}. It states that \textit{when one approximates a continuum domain by a metric graph, then the FPS of a metric graph GFF converges in law to the FPS of a continuum GFF, for the Hausdorff distance}. This result holds for finitely-connected domains, and for piece-wise constant boundary conditions.

In fact, Proposition  \ref{Convergence} shows that the coupling between the GFF and the FPS converges. The proof relies on the characterization of the FPS in continuum as the unique local set such that the GFF restricted to it is a positive measure, and outside is a conditional independent GFF with boundary values equal to $-a$ \cite{ALS1}. It is accompagned by a convergence result on the clusters of the metric graph loop soup that contain at least one boundary-to-boundary excursion (Proposition \ref{PropConvClustExc}).

Together, these convergence results have numerous interesting implications, whose study takes up most of this paper. Let us first mention a family of convergence results: certain natural interfaces in the metric graph GFFs on a 2D lattice converge to level lines of the continuum GFF (Proposition \ref{ConvSLE}). These results are reminiscent on Schramm-Sheffield's convergence of the zero level line of 2D discrete GFF to SLE$_{4}$ \cite{SchSh,SchSh2}.

Let us remark that Proposition \ref{ConvSLE} does not cover the results of \cite{SchSh2}, as the interfaces we deal with do not appear at the level of the discrete GFF. Yet, the discrete interfaces we consider are as natural, and the proofs for the convergence are way simpler. In particular, we show that if we consider metric graph GFFs on a lattice approximation of $\D$, with boundary conditions $-\lambda$ on the left half-circle and $\lambda$ on the right half-circle, then the left boundary of the FPS of level $-\lambda$ converges to the Schramm-Sheffield level line, and thus to a SLE$_{4}$ curve w.r.t. the Hausdorff distance (Corollary \ref{convSLE2}).

Several other central consequences of the FPS convergence have to do with isomorphism theorems. In general, the isomorphism theorems relate the square of a GFF, discrete or continuum if the dimension is less or equal to $3$, to occupation times of Markovian trajectories, Markov jump processes in discrete case, Brownian motions in continuum. Originally formulated by Dynkin \cite{Dynkin1983MarkovFieldTh,Dynkin1984Isomorphism,
	Dynkin1984IsomorphismPresentation}, there are multiple versions of them \cite{Eisenbaum1995Iso, EKMRS2000RK, Sznitman2012Isomorphism, LeJan2007Unpub, LeJan2011Loops}, see also
\cite{MarcusRosen2006MarkovGaussianLocTime,Sznitman2012LectureIso} for reviews.
For instance, in Le Jan's isomorphism \cite{LeJan2007Unpub, LeJan2011Loops}, the whole square of a discrete GFF is given by the occupation field of a Markov jump process loop-soup.
The introduction of metric graphs as in \cite{Lupu2016Iso} provides "polarized" versions of isomorphism theorems, where one has the additional property that the GFF has constant sign on each Markovian trajectory. More precisely, one considers a metric graph loop-soup and an independent Poisson point process of boundary-to-boundary metric graph excursions. Among all the clusters formed by these trajectories, one takes those that contain at least an excursion, that is to say are connected to the boundary. Then, the closed union of such clusters is distributed as a metric graph FPS (Proposition \ref{CorFPCluster}).

As consequence of the convergence results, this representation of the FPS transfers to the continuum. In other words, \textit{the continuum FPS can be represented as a union of clusters of two-dimensional Brownian loops (out of a critical Brownian loop-soup as in \cite{LW2004BMLoopSoup}, of central charge $c=1$) and Brownian boundary-to-boundary excursions (Proposition \ref{CorEqFPS}).}
This description can be viewed as a non-perturbative version of Symanzik's loop expansion in Euclidean QFT
\cite{Symanzik66Scalar,Symanzik1969QFT} (see also \cite{BFS82Loop}).

In Proposition \ref{PropFPSplusWick}, we combine our description of the FPS by loops and excursions with 
the renormalized Le Jan's isomorphism \cite{LeJan2010LoopsRenorm}, formulated in terms of renormalized (Wick) square of the GFF, and the renormalized centered occupation field of loops and excursions. In this way, we get the square and the interfaces of the GFF on the same picture. In the simply-connected case with zero boundary conditions, one can further ask these interfaces to correspond to the Miller-Sheffield coupling of CLE$_4$ and the GFF \cite{MS}, \cite{ASW}. This implies that conditional on its outer boundary, the law of a Brownian loop cluster of central charge $c=1$ is that of a first passage set of level $-2\lambda$ (Corollary \ref{CorClusterLoopSoup}), extending the results of \cite{QW2015}.

A natural question which arose in view of this new isomorphism was how to take the "square root" of the Wick square of the continuum GFF in order to access the value of the GFF on the FPS. This was solved in \cite{ALS1}, going through the Liouville quantum gravity (Gaussian multiplicative chaos). The "square root" turned to be a Minkowski content measure of the FPS in the
gauge $r\mapsto |\log r|^{1/2}r^2$. This also gives, via the isomorphism, the right gauge for measuring the size of clusters in a critical Brownian loop-soup ($c=1$). For subcritical Brownian loop-soups ($c<1$), the gauge is still unknown.

We draw additional consequences from the isomorphism theorems in Section \ref{SubSecConseq1} - we show local finiteness of the FPS, prove that its a.s. Hausdorff dimension is 2 and that it satisfies a Harris-FKG inequality. In Section \ref{SubSecConseq2}, we also study more general families of level lines, and show for example that the multiple commuting SLE$_{4}$ \cite{Dub07MultipleSLE} are envelopes of Brownian loops and excursions (Corollary \ref{CorLvlLineCluster} and Remark \ref{RemLvlLine}). Previously, similar results were known only for single SLE$_{\kappa}(\rho)$ processes 
\cite{WernerWu2013FromCLEtoSLE}
and the conformal loop ensembles CLE$_{\kappa}$ 
(loop-soup construction \cite{SheffieldWerner2012CLE}). Finally, in Corollary
\ref{CorCoupling} how to construct explicit coupling of Gaussian free fields with different boundary conditions such that some level lines coincide with positive probability.

In a follow-up paper \cite{ALS4}, we will use the techniques developed here to define an excursion decomposition of the GFF.

The rest of this paper is structured as follows.

In Section \ref{SecMetricG}, we recall the construction of the GFF on the metric graph and the related isomorphims. We also recall the definition of the first passage set on metric graph and its construction out of metric graph loops and excursions.

Section \ref{sectC} is devoted to preliminaries on the continuum GFF, its first passage sets, and Le Jan's isomorphism representing the Wick's square of the GFF as centered occupation field of a Brownian loop-soup. In particular, we extend Le Jan's isomorphism to the GFF with positive variable boundary conditions by introducing boundary to boundary excursions.

In Section \ref{SecConv}, we first introduce the notions of convergence of domains, fields, compact sets, trajectories we use. Then, we show the convergence of metric graph FPS to continuum FPS; and the convergence of metric graph loops and excursions to clusters of 2D Brownian loops and excursions.

Finally, in Section \ref{SecConsec} we derive several consequences of the convergence results, including the identification of  the continuum FPS with clusters of Brownian loops and excursions in Section \ref{SubSecConseq1} and the convergence of certain natural FPS interfaces to the SLE$_4$ curves in Section \ref{SubSecConseq2}.

%%%%%%%%%%%%%%%%%%%%%%%%%%%%%%%%%%%%%%%%%%%%%%%%%%%%%%%%%%%%%%%

\section{Preliminaries on the metric graph}
\label{SecMetricG}

In this section, we first give the definition of the metric graph and define the GFF on top of it - basically it corresponds to taking a discrete GFF on its vertices, and extending it using conditionally independent Brownian bridges of length equal to the resistance on all edges. Next, we browse through the measures on loops and excursions on the metric GFF; and the isomorphism theorems. In Section \ref{dFPS}, we define the first passage set (FPS) of the metric graph introduced in \cite{LupuWerner2016Levy} and bring out its representation using Brownian loops and excursions. 

The results in this section are either already in the literature or are slight extensions of already existing results. For example, we extend the isomorphism theorems on the metric graph to non-constant boundary conditions.

\subsection{The Gaussian free field on metric graphs}
\label{SubSecGFFMetric}

We start from a finite connected undirected graph $\mathcal{G}=(V,E)$  with no multiple edges or self-loops. We interpret it is as an electrical network by equipping each edge $e=\lbrace x,y\rbrace\in E$ with a conductance $C(e)=C(x,y)>0$ . 
If $x,y\in V$, $x\sim y$ denotes that $x$ and $y$ are connected by an edge.  
A special subset of vertices $\partial\mathcal{G}\subset V$  will be considered as the boundary of the network.
We assume that $\partial\mathcal{G}$ and $V\backslash\partial\mathcal{G}$ are non-empty. 
For $x\in V\backslash\partial\mathcal{G}$, we denote
\begin{displaymath}
C_{\rm tot}(x):=\sum_{\substack{y\in V\\y\sim x}}C(x,y).
\end{displaymath}

Let $\Delta^{\mathcal{G}}$ be the discrete Laplacian:
\begin{displaymath}
(\Delta^{\mathcal{G}}f)(x):=
\sum_{y\sim x}C(x,y)(f(y)-f(x)).
\end{displaymath}
Let $\mathcal{E}_{\mathcal{G}}$   be the Dirichlet energy:
\begin{displaymath}
\mathcal{E}_{\mathcal{G}}(f,f):=
-\sum_{x\in V}\sum_{y\sim x}f(x)(\Delta^{\mathcal{G}} f)(x)
=\sum_{\lbrace x,y\rbrace\in E}
C(x,y)(f(y)-f(x))^{2}.
\end{displaymath}

Let $\phi$  be the discrete Gaussian free field (GFF) on 
$\mathcal{G}$, associated to the Dirichlet energy $\mathcal{E}_{\mathcal{G}}$, with boundary condition $0$. That is to say, if we defined the Green's function
$G_{\mathcal{G}}$  as the inverse of $-\Delta^{\mathcal G}$, with $0$ boundary conditions on $\partial \mathcal G$, we have that 
$\phi$ is the only centred Gaussian process such that for any $f,g:V\mapsto \R$
\begin{align*}
\E\left[(\phi,f_{1})(\phi,f_{2}) \right]= \sum_{x,y \in V\backslash \partial \mathcal G} f_{1}(x) G_{\mathcal G}(x,y) f_{2}(y).
\end{align*}

We would sometimes be interested in a GFF with non-0 boundary conditions. For that we call $u: V\mapsto \mathbb \R$  a boundary condition if it is harmonic function in $V\backslash \partial \mathcal G$ and when the context is clear we identify it with its restriction to $\partial \mathcal G$. Now note that
$\phi+u$ is then the GFF with boundary condition $u$. Its expectation is $u$ and its covariance is given by the Green's function.

Given an electrical network $\mathcal{G}$, we can associate to it
a \textit{metric graph}, also called \textit{cable graph} or 
\textit{cable system}, denoted $\widetilde{\mathcal{G}}$ .
Topologically, it is a simplicial complex of degree $1$, where each edge is replaced by a continuous line segment. We also endow each such segment with a metric such that its length is equal to the resistance $C(x,y)^{-1}$, $x$ and $y$ being the endpoints. 
One should think of it as replacing a ``discrete'' resistor by a ``continuous'' electrical wire, where the resistance is proportional to the length.

Given a discrete GFF $\phi$ with boundary condition $0$, we interpolate it to a function on 
$\widetilde{\mathcal{G}}$  by adding on each 
edge-line a conditionally independent standard Brownian bridge. If the line joins the vertices $x$ and $y$, the endvalues of the bridge
would be $\phi(x)$ and $\phi(y)$, and its length $C(x,y)^{-1}$. By doing that we get a continuous function $\tilde{\phi}$ on 
$\widetilde{\mathcal{G}}$ (Figure \ref{FigMetricGFF}). This is precisely the metric graph GFF with $0$ boundary conditions. Consider the linear interpolation of $u$ inside the edges, still denoted by $u$. $\tilde{\phi}+u$ is the metric graph GFF with boundary conditions $u$. The restriction of $\tilde{\phi}+u$ to the vertices is the discrete GFF $\phi+u$.

\begin{figure}[ht!]
	\centering
	\includegraphics[width=3.2in]{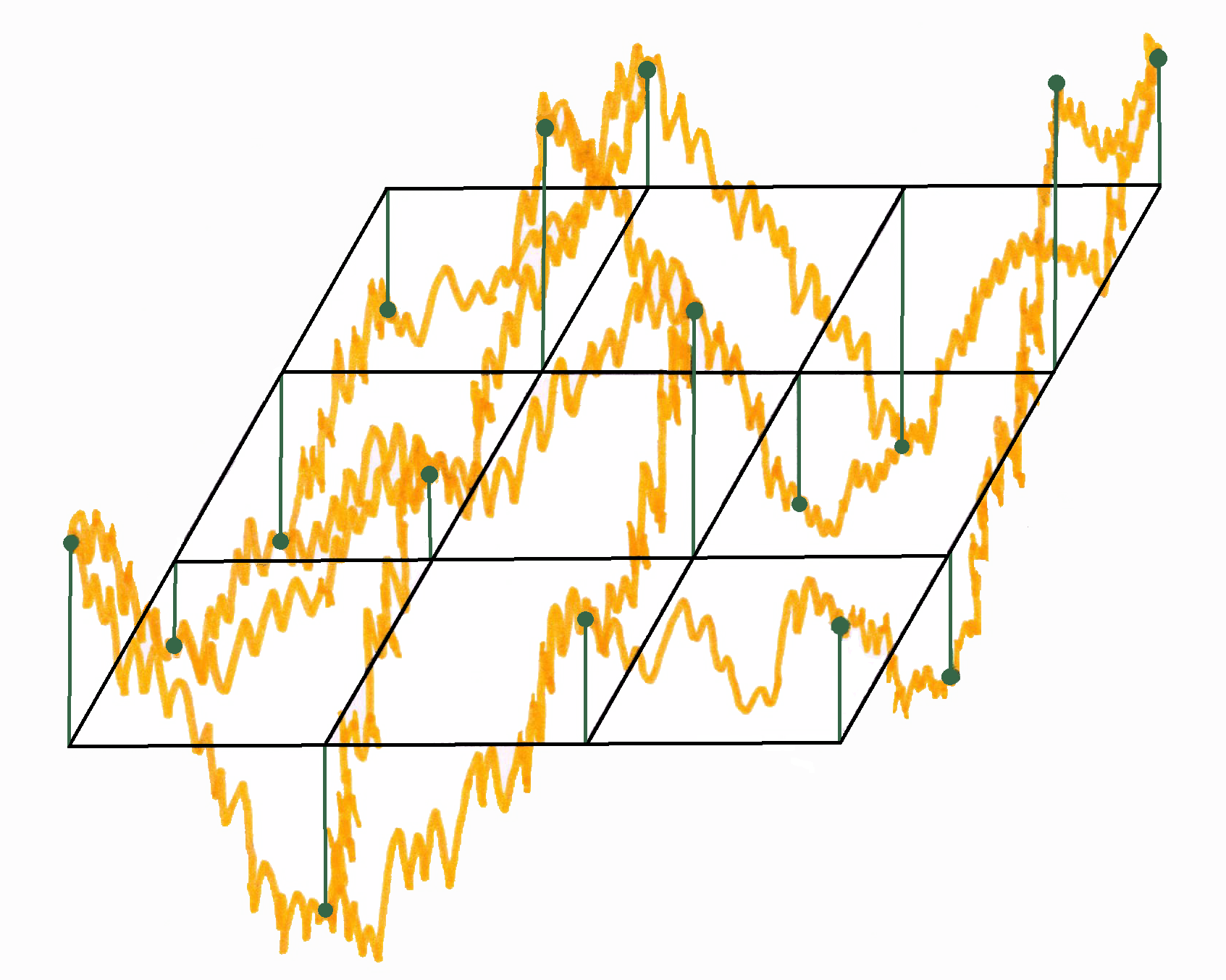}
	\caption{$\tilde{\phi}$ on a square lattice. Green dots represent the values of the discrete GFF. Orange curves are the Brownian bridges interpolating between these values.}
	\label{FigMetricGFF}
\end{figure}

The metric graph GFF satisfies the strong Markov property on 
$\widetilde{\mathcal{G}}$. More precisely, assume that $A$ is a random compact subset of $\mathcal{G}$.
We say that is \textit{optional} for $\tilde{\phi}$ if for every 
$O$ deterministic open subset of $\mathcal{G}$,
the event $A\subseteq O$
is measurable with respect the restriction of $\tilde{\phi}$ to 
$O$.
For simplicity we will also assume that a.s., $A$ has finitely many connected components. Then 
$\widetilde{\mathcal{G}}\backslash A$ has finitely many connected components too, and the closure of each connected component is a metric graph, even if an edge of $\mathcal{G}$ is split among several connected components or partially covered by $A$.
\begin{prop}[Strong Markov property, \cite{Lupu2016Iso}]
	\label{PropStrongMarkov}
	Let $A$ be a random compact subset of 
	$\widetilde{\mathcal{G}}$, with finitely many connected components and
	optional for the metric graph GFF $\tilde{\phi}$. Then
	we have a Markov decomposition
	\begin{displaymath}
	\tilde{\phi}=\tilde{\phi}_{A}+\tilde{\phi}^{A},
	\end{displaymath}
	where, conditionally on $A$, $\tilde{\phi}^{A}$ is a
	zero boundary metric graph GFF on 
	$\widetilde{\mathcal{G}}\backslash A$ independent of
	$\tilde{\phi}_{A}$ (and by convention zero on $A$), and 
	$\tilde{\phi}_{A}$ is on $A$ the restriction of $\tilde{\phi}$ to $A$ and on $\widetilde{\mathcal{G}}\backslash A$ equals a harmonic function 
	$\tilde h_{A}$, whose boundary values are given by 
	$\tilde{\phi}$ on $\partial\mathcal{G}\cup A$.
\end{prop}

\subsection{Measures on loops and excursions} \label{Loops Metric}

Next, we introduce the measures on loops and boundary-to-boundary excursions which appear in isomorphism theorems in discrete and metric graph settings.

Consider the nearest neighbour Markov jump process on $\mathcal{G}$, with jump rates given by the conductances, and let 
$p^{\mathcal{G}}_{t}(x,y)$ and 
$\mathbb{P}^{\mathcal{G},x,y}_{t}$ be the associated 
transition probabilities and bridge probability measures respectively.
Let $T_{\partial\mathcal{G}}$ be the first time the jump process hits the boundary $\mathcal{G}$. 
The \textit{loop measure} on $\mathcal{G}$ is defined to be
\begin{equation}
\label{DefMeasLoopDiscr}
\mu^{\mathcal{G}}_{\rm loop}(\cdot):=
\sum_{x\in V}\int_{0}^{+\infty}
\mathbb{P}^{\mathcal{G},x,x}_{t}(\cdot,T_{\partial\mathcal{G}}>t)
p^{\mathcal{G}}_{t}(x,x)\dfrac{dt}{t}.
\end{equation}
$\mu^{\mathcal{G}}_{\rm loop}$ is a measure on nearest neighbour paths in $V\backslash \partial\mathcal{G}$, parametrized
by continuous time, which at the end return to the starting point. Note that it associates an infinite mass to trivial loops, which only stay at one given vertex. This measure was introduced by 
Le Jan in \cite{LeJan2007Unpub, LeJan2010LoopsRenorm, LeJan2011Loops}. If one restricts the measure to non-trivial loops and forgets the time-parametrisation, one gets the measure on random walk loops which appears in \cite{LawlerFerreras2007RWLoopSoup,LawlerLimic2010RW}.

$\Gamma$ will denote the family of all finite paths parametrized by discrete time, which start and end in $\partial\mathcal{G}$, only visit $\partial\mathcal{G}$
at the start and at the end, and also visit 
$V\backslash\partial\mathcal{G}$. We see a path in $\Gamma$
as the skeleton of an excursion from $\partial\mathcal{G}$ to itself.
We introduce a measure $\nu^{\mathcal{G}}_{\rm exc}$ on $\Gamma$ as follows. The mass given to an admissible path 
$(x_{0},x_{1},\dots,x_{n})$ is
\begin{displaymath}
\prod_{i=1}^{n}C(x_{i-1},x_{i})
\prod_{i=1}^{n-1}C_{\rm tot}(x_{i})^{-1}.
\end{displaymath}
Note that this measure is invariant under time-reversal.
For $x,y\in\partial\mathcal{G}$,
$\Gamma_{x,y}$ will denote the subset of $\Gamma$ made of paths
that start at $x$ and end at $y$. We defined the kernel
$H_{\mathcal{G}}(x,y)$  on 
$\partial\mathcal{G}\times \partial\mathcal{G}$ as
\begin{displaymath}
H_{\mathcal{G}}(x,y):=\nu^{\mathcal{G}}_{\rm exc}(\Gamma_{x,y}).
\end{displaymath}
It is symmetric.
$H_{\mathcal{G}}$ is often referred to as the \textit{discrete boundary Poisson kernel}, and this is the terminology we will use.
$\mathbb{P}_{\rm exc}^{\mathcal{G},x,y}$ will denote the probability measure on excursions from $x$ to $y$ parametrized by continuous time.
The discrete-time skeleton of the excursion is distributed
according to the probability measure
$\1_{\Gamma_{x,y}}H_{\mathcal{G}}(x,y)^{-1}\nu^{\mathcal{G}}_{\rm exc}$. 
The excursions under
$\mathbb{P}_{x,y}^{\rm exc}$ spend zero time at $x$ and $y$, i.e. they immediately jump away from $x$ and jump to $y$ at the last moment.
Conditionally on the skeleton $(x_{0},x_{1},\dots,x_{n})$, the holding time at $x_{i}$, $1\leq i\leq n-1$, is distributed as an exponential
r.v. with mean $C_{\rm tot}(x_{i})^{-1}$, and all the holding times are conditionally independent. To a non negative boundary condition 
$u$ on $\partial \mathcal{G}$ we will associate the measure
\begin{equation}
\label{DefMeasExcDisc}
\mu^{\mathcal{G},u}_{\rm exc}:=
\dfrac{1}{2}\sum_{(x,y)\in\partial\mathcal{G}\times \partial\mathcal{G}}
u(x)u(y)H_{\mathcal{G}}(x,y)\mathbb{P}_{\rm exc}^{\mathcal{G},x,y}.
\end{equation}

Consider now the metric graph setting. We will consider on 
$\widetilde{\mathcal{G}}$ a diffusion we introduce now. For generalities on diffusion processes on metric graphs, see
\cite{BaxterChacon1984DiffusionsNetworks,EnriquezKifer2001BMGraphs}.
$(\widetilde{X}_{t})_{t\geq 0}$ will be a Feller process on
$\widetilde{\mathcal{G}}$. The domain of its infinitesimal generator
$\Delta^{\widetilde{\mathcal{G}}}$ will contain all continuous functions which are 
$\mathcal{C}^{2}$ inside each edge and such that the second derivatives have limits at the vertices and which are the same for every adjacent edge. On such a function $f$, $\Delta^{\widetilde{\mathcal{G}}}$ will
act as $\Delta^{\widetilde{\mathcal{G}}}f=f''/4$, i.e. one takes the second derivative inside each edge. $\widetilde{X}$ behaves inside an edge like a one-dimensional Brownian motion. With our normalization of 
$\Delta^{\widetilde{\mathcal{G}}}$, it is not a standard Brownian motion, but with variance multiplied by $1/2$. When $\widetilde{X}$
hits an edge of degree $1$, it behaves like a reflected Brownian motion near this edge. When it hits an edge of degree $2$, it behaves just like a Brownian motion, as we can always consider that the two lines associated to the two adjacent edges form a single line.
When $\widetilde{X}$ hits a vertex of degree at least three, then it performs Brownian excursions inside each adjacent edge, until hitting an
neighbouring vertex. Each adjacent edge will be visited infinitely many times immediately when starting from a vertex, and there is no notion of first visited edge. The rates of small excursions will be the same for each adjacent edge. See 
\cite{Lupu2016Iso,EnriquezKifer2001BMGraphs} for details.

Just as a one-dimensional Brownian motion, 
$(\widetilde{X}_{t})_{t\geq 0}$ has local times.
Denote $\tilde{m}$ the measure on $\widetilde{\mathcal{G}}$ such that
its restriction to each edge-line is the Lebesgue measure.
There is a family of local times
$(L^{x}_{t}(\widetilde{X}))_{x\in\widetilde{\mathcal{G}},t\geq 0}$,
adapted to the filtration of $(\widetilde{X}_{t})_{t\geq 0}$ and jointly continuous in $(x,t)$, such that for any $f$ measurable bounded function on $\widetilde{\mathcal{G}}$,
\begin{displaymath}
\int_{0}^{t}f(\widetilde{X}_{s}) ds =
\int_{\widetilde{\mathcal{G}}}f(x)
L^{x}_{t}(\widetilde{X}) d\tilde{m}(x).
\end{displaymath}
On should note that in particular the local times are space-continuous at the vertices. See \cite{Lupu2016Iso}. Consider the continuous
additive functional (CAF)
\begin{equation}
\label{EqCAF}
(t,(\widetilde{X})_{0\leq s\leq t})
\mapsto \sum_{x\in V}L^{x}_{t}(\widetilde{X}).
\end{equation}
It is constant outside the times $\widetilde{X}$ spends at vertices. By performing a time change by the inverse of the CAF \eqref{EqCAF}, one gets a continuous-time paths on the discrete network
$\mathcal{G}$ which jumps to the nearest neighbours. It actually has the same law as the Markov jump process on $\mathcal{G}$ with the rates of jumps given by the conductances. See \cite{Lupu2016Iso}.

The process $(\widetilde{X}_{t})_{t\geq 0}$ has transition densities and bridge probability measures, which we will denote
$p^{\widetilde{\mathcal{G}}}_{t}(x,y)$ and 
$\mathbb{P}^{\widetilde{\mathcal{G}},x,y}_{t}$ respectively.
$\widetilde{T}_{\partial \mathcal{G}}$ will denote the first time
$(\widetilde{X}_{t})_{t\geq 0}$ hits the boundary 
$\partial \mathcal{G}$. The loop measure on the metric graph
$\widetilde{\mathcal{G}}$ is defined to be 
\begin{displaymath}\label{Def Loop Met}
\mu^{\widetilde{\mathcal{G}}}_{\rm loop}(\cdot):=
\int_{\widetilde{\mathcal{G}}}\int_{0}^{+\infty}
\mathbb{P}^{\widetilde{\mathcal{G}},x,x}_{t}
(\cdot,\widetilde{T}_{\partial\mathcal{G}}>t)
p^{\widetilde{\mathcal{G}}}_{t}(x,x)\dfrac{dt}{t}d\tilde{m}(x).
\end{displaymath}
It has infinite total mass.
This definition is the exact analogue of the definition
\eqref{DefMeasLoopDiscr} of the measure on loops on discrete network
$\mathcal{G}$. Under the measure 
$\mu^{\widetilde{\mathcal{G}}}_{\rm loop}$, the loops do not hit the boundary $\widetilde{\mathcal{G}}$. One can almost recover
$\mu^{\mathcal{G}}_{\rm loop}$ from
$\mu^{\widetilde{\mathcal{G}}}_{\rm loop}$. Just as the process 
$(\widetilde{X}_{t})_{t\geq 0}$ itself, the loops under
$\mu^{\widetilde{\mathcal{G}}}_{\rm loop}$ admit a continuous family of local times. One can consider the CAF \eqref{EqCAF} applied to a metric graph loop $\tilde{\gamma}$ that visits at least one vertex.
By performing the time-change by the inverse of this CAF, one gets a nearest neighbour loop on the discrete network $\mathcal{G}$.
The image by this map of the measure 
$\mu^{\widetilde{\mathcal{G}}}_{\rm loop}$, restricted to the loops that visit at least one vertex,
is $\mu^{\mathcal{G}}_{\rm loop}$, up to a change of root 
(i.e. starting and endpoint) of the discrete loop. So, if one rather considers the unrooted loops and the measures projected on the quotients, then one obtains $\mu^{\mathcal{G}}_{\rm loop}$ as the image of $\mu^{\widetilde{\mathcal{G}}}_{\rm loop}$ by a change of time.
Moreover, the holding times at vertices of discrete network loops are equal to the increments of local times at vertices of metric graph loops between two consecutive edge traversals.
Note that 
$\mu^{\widetilde{\mathcal{G}}}_{\rm loop}$ also puts mass on the loops that do not visit any vertex. These loops do not matter for
$\mu^{\mathcal{G}}_{\rm loop}$.
See \cite{FitzsimmonsRosen2012LoopsIsomorphism} for generalities on the covariance of the measure on loops by time change by an inverse of a CAF.

On the metric graph one also has the analogue of the measure
$\mu^{\mathcal{G},u}_{\rm exc}$ on excursions from boundary to boundary defined by \eqref{DefMeasExcDisc}. Let $x\in \partial\mathcal{G}$ and let $k$ be the degree of $x$. Let $\varepsilon>0$ be smaller than the smallest length of an edge adjacent to $x$. $x_{1,\varepsilon},\dots,x_{k,\varepsilon}$ will denote
the points inside each of the adjacent edge to $x$ which are located at distance $\varepsilon$ from $x$. The measure on excursions from $x$ to the boundary is obtained as the limit
\begin{displaymath}
\mu^{\widetilde{\mathcal{G}},x}_{\rm exc}(F(\tilde{\gamma}))=
\lim_{\varepsilon\rightarrow 0}
\varepsilon^{-1}\sum_{i=1}^{k}
\mathbb{E}_{x_{i,\varepsilon}}
\left[F((\widetilde{X}_{t})
_{0\leq t\leq \widetilde{T}_{\partial\mathcal{G}}})\right],
\end{displaymath}
where $F$ is any measurable bounded functional on paths. If
$y\in\partial\mathcal{G}$ is another boundary point, possibly the same,
$\mu^{\widetilde{\mathcal{G}},x,y}_{\rm exc}$ will denote the restriction of $\mu^{\widetilde{\mathcal{G}},x}_{\rm exc}$ to
excursions that end at $y$. 
$\mu^{\widetilde{\mathcal{G}},y,x}_{\rm exc}$ is the image of
$\mu^{\widetilde{\mathcal{G}},x,y}_{\rm exc}$ by time-reversal.
If $y\neq x$, $\mu^{\widetilde{\mathcal{G}},x,y}_{\rm exc}$ has a finite mass, which equals $H_{\mathcal{G}}(x,y)$. To the contrary, the mass of $\mu^{\widetilde{\mathcal{G}},x,x}_{\rm exc}$
is infinite. However, the restriction of 
$\mu^{\widetilde{\mathcal{G}},x,x}_{\rm exc}$ to excursions that visit 
$V\backslash\partial\mathcal{G}$ has a finite mass equal to $H_{\mathcal{G}}(x,x)$. 

Given $u$ a non-negative boundary condition on 
$\partial \mathcal{G}$, we define the following measure on excursions from boundary to boundary on the metric graph:
\begin{displaymath}
\mu^{\widetilde{\mathcal{G}},u}_{\rm exc}=
\dfrac{1}{2}\sum_{(x,y)\in\partial\mathcal{G}\times\partial\mathcal{G}}
u(x)u(y)\mu^{\widetilde{\mathcal{G}},x,y}_{\rm exc}.
\end{displaymath}
If one restricts $\mu^{\widetilde{\mathcal{G}},u}_{\rm exc}$ to excursions that visit $V\backslash\partial\mathcal{G}$ and performs on these excursions the time-change by the inverse of the CAF
\eqref{EqCAF}, one gets a measure on discrete-space continuous-time
boundary-to-boundary excursions which is exactly 
$\mu^{\mathcal{G},u}_{\rm exc}$.
Particular cases of above metric graph excursion measures were used in
\cite{Lupu2015ConvCLE}.

Next we state a Markov property for the metric graph excursion measure
$\mu^{\widetilde{\mathcal{G}},x}_{\rm exc}$. Let $K$ be a compact connected subset of $\widetilde{\mathcal{G}}$. The boundary $\partial K$ of $K$ will be by definition the union of the topological boundary of 
$K$ as a subset of $\widetilde{\mathcal{G}}$ and 
$\partial\mathcal{G}\cap K$. $K$ is a metric graph itself. Its set of vertices is $(V\cap K)\cup\partial K$. If an edge of $\mathcal{G}$ is entirely contained inside $K$, it will be an edge of $K$ and it will have the same conductance. $K$ can also contain one or two disjoint subsegments of an edge of $\mathcal{G}$. Each subsegment is a (different) edge for $K$, and the corresponding conductances are given by the inverses of the lengths of subsegments. So $K$ is naturally endowed with a boundary Poisson kernel
$(H_{K}(x,y))_{x,y\in \partial K}$ 
and boundary-to-boundary excursion measures
$(\mu^{K,x,y}_{\rm exc})_{x,y\in \partial K}$. Note that these objects depend only on $K$ and $\partial K$, and not on how $K$ is embedded in
$\widetilde{\mathcal{G}}$. 
\begin{prop}
	\label{PropMarkovExcMG}
	Let $x\in\partial\mathcal{G}$, and $K$ a compact connected subset of the metric graph $\widetilde{\mathcal{G}}$ which contains $x$ and such that
	$\widetilde{\mathcal{G}}\backslash K\neq\emptyset$. Denote by 
	$\gamma_{1}\circ\gamma_{2}$ the concatenation of paths
	$\gamma_{1}$ and $\gamma_{2}$, where $\gamma_{1}$ comes first.
	For any $F$ bounded measurable functional on paths, we have
	\begin{displaymath}
	\mu^{\widetilde{\mathcal{G}},x}_{\rm exc}(F(\gamma),
	\gamma~\text{visits}~\widetilde{\mathcal{G}}\backslash K)=
	\sum_{y\in\partial K\backslash\partial\mathcal{G}}
	\mu^{K,x,y}_{\rm exc}\otimes\mathbb{E}_{y}
	\left[F(\gamma_{1}\circ(\widetilde{X}_{t})_{0\leq t\leq 
		\widetilde{T}_{\partial\mathcal{G}}})\right],
	\end{displaymath}
	where $\mathbb{E}_{y}$ stands for the metric graph Brownian motion
	$\widetilde{X}$ inside $\widetilde{\mathcal{G}}$, started from $y$.
\end{prop}

\subsection{Isomorphism theorems}

The continuous time random walk \textit{loop-soup} 
$\mathcal{L}^\mathcal{G}_{\alpha}$ is a Poisson point process (PPP) of intensity $\alpha \mu^{\mathcal{G}}_{\rm loop}$, $\alpha >0$.
We view it as a random countable collection of loops. We will also consider PPP-s of boundary-to-boundary excursions
$\Xi_{u}^{\mathcal{G}}$, of intensity $\mu^{\mathcal{G},u}_{\rm exc}$,
where $u:\partial\mathcal{G}\rightarrow\mathbb{R}_{+}$ is a non-negative boundary condition.

The \textit{occupation field} of a path $(\gamma(t))_{0\leq t\leq t_{\gamma}}$ in $\mathcal{G}$, parametrized by continuous time, is
\begin{displaymath}
L^{x}(\gamma)=\int_{0}^{t_{\gamma}}\1_{\gamma(t)=x} dt.
\end{displaymath}
The occupation field of a loop-soup $\mathcal{L}^\mathcal{G}_{\alpha}$ is
\begin{displaymath}
L^{x}(\mathcal{L}^\mathcal{G}_{\alpha})=
\sum_{\gamma\in \mathcal{L}^\mathcal{G}_{\alpha}}L^{x}(\gamma).
\end{displaymath}
Same definition for the occupation field of $\Xi_{u}^{\mathcal{G}}$. At the intensity parameter $\alpha=1/2$, these occupation fields are relate to the square of GFF:

\begin{prop}
	\label{PropIsoDiscr}
	Let $u:\partial\mathcal{G}\rightarrow\mathbb{R}_{+}$ be a non-negative boundary condition. Take $\mathcal{L}^\mathcal{G}_{1/2}$ and 
	$\Xi_{u}^{\mathcal{G}}$ independent. Then, the sum of occupation fields
	\begin{displaymath}
	\left(L^{x}(\mathcal{L}^\mathcal{G}_{1/2})
	+L^{x}(\Xi_{u}^{\mathcal{G}})\right)
	_{x\in V\backslash\partial\mathcal{G}}
	\end{displaymath}
	is distributed like
	\begin{displaymath}
	\left(\dfrac{1}{2}(\phi+u)^{2}(x)\right)
	_{x\in V\backslash\partial\mathcal{G}},
	\end{displaymath}
	where $\phi+u$ is the GFF with boundary condition $u$.
\end{prop}

\begin{proof}
	If $u\equiv 0$, there are no excursions we are in the setting of Le Jan's isomorphism for loop-soups (\cite{LeJan2007Unpub,LeJan2011Loops}). If $u$ is constant and strictly positive, then the proposition follows by combining Le Jan's isomorphism and the generalized second Ray-Knight theorem (\cite{MarcusRosen2006MarkovGaussianLocTime,Sznitman2012LectureIso}). Indeed, then one can consider the whole boundary 
	$\partial\mathcal{G}$ as a single vertex, and the boundary to boundary excursions as excursions outside this vertex.
	
	The case of $u$ non-constant can be reduced to the previous one. We first assume that $u$ is strictly positive on $\partial\mathcal{G}$.
	The general case can be obtained by taking the limit. 
	We define new conductances on the edges:
	\begin{displaymath}
	\widehat{C}(x,y):=C(x,y)u(x)u(y),
	\end{displaymath}
	where $x$ and $y$ are neighbours in $\mathcal{G}$. Let
	$\hat{\phi}$ be the 0 boundary GFF associated to the new conductances
	$\widehat{C}$.
	We claim that
	\begin{displaymath}
	(\hat{\phi}(x))_{x\in V}
	\stackrel{(d)}{=}
	(u(x)^{-1}\phi(x))_{x\in V}.
	\end{displaymath}
	To check the identity in law one has to check the identity of energy functions:
	\begin{eqnarray*}
		&&\mathcal{E}_{\mathcal{G}}(uf,uf)=
		-\sum_{x\in V}\sum_{y\sim x}
		u(x)f(x)C(x,y)(u(y)f(y)-u(x)f(x))\\
		&&=
		-\sum_{x\in V}\sum_{y\sim x}
		u(x)f(x)C(x,y)u(y)(f(y)-f(x))
		+\sum_{x\in V}\sum_{y\sim x}
		C(x,y)(u(y)-u(x))u(x)f(x)^{2}\\
		&&=-\sum_{x\in V}\sum_{y\sim x}
		f(x)\widehat{C}(x,y)(f(y)-f(x))+
		\sum_{x\in\partial\mathcal{G}}\sum_{y\sim x}
		C(x,y)(u(y)-u(x))u(x)f(x)^{2}\\
		&&=\widehat{\mathcal{E}}(f,f)+0.
	\end{eqnarray*}
	From the second to the third line we used that $u$ is harmonic.
	
	Now, we can apply the case of constant boundary conditions to
	$\frac{1}{2}(\hat{\phi}+1)^{2}$. We get that it is distributed like the occupation field of a loop-soup of parameter $\alpha=1/2$ and	
	an independent Poissonian family of excursions 	from $\hat{x}$ to $\hat{x}$, both associated to the jump rates $\widehat{C}(x,y)$. If on these paths we perform the time change
	\begin{equation}
	\label{EqTimeChange}
	dt=u(x)^{-2} ds,
	\end{equation}	
	we get $\mathcal{L}^\mathcal{G}_{1/2}$ and 	$\Xi_{u}^{\mathcal{G}}$.
	The time change \eqref{EqTimeChange} multiplies the occupation field by $u^{2}$, which exactly transforms $(\hat{\phi}+1)^{2}$ into
	$(\phi+u)^{2}$.
\end{proof}

Note that the coupling
$(L(\mathcal{L}^{\mathcal{G}}_{1/2}),
L(\mathcal{L}^{\mathcal{G}}_{1/2})+
L(\Xi^{\mathcal{G}}_{u}))$ is not the same as
$(\frac{1}{2}\phi^{2}, \frac{1}{2}(\phi+u)^{2})$.

On a metric graph, the isomorphism given by Proposition
\ref{PropIsoDiscr} still holds. But in this setting one has a stronger version of it, which takes in account the sign of the GFF.
Consider a PPP of loops (loop-soup)
$\mathcal{L}^{\widetilde{\mathcal{G}}}_{1/2}$ on the metric graph 
$\widetilde{\mathcal{G}}$, of intensity
$\frac{1}{2}\mu^{\widetilde{\mathcal{G}}}_{\rm loop}$, 
and an independent PPP of metric graph excursions
from boundary to boundary, $\Xi_{u}^{\widetilde{\mathcal{G}}}$, of intensity 
$\mu^{\widetilde{\mathcal{G}},u}_{\rm exc}$. 
For $x\in \widetilde{\mathcal{G}}$, 
$L^{x}(\mathcal{L}^{\widetilde{\mathcal{G}}}_{1/2})$ is defined as the sum over the loops of the local time at $x$ accumulated by the loops. The occupation field 
$L^{x}(\Xi_{u}^{\widetilde{\mathcal{G}}})$ is defined similarly.
$L^{x}(\Xi_{u}^{\widetilde{\mathcal{G}}})$ is a locally finite sum,
except at the boundary points $\partial\mathcal{G}$, but there it converges to 
$\frac{1}{2} u^{2}$. Indeed, for this limit only matter the excursions that do not visit $V\backslash\partial\mathcal{G}$, but then we are in the case of excursions of a one-dimensional Brownian motion.
To the contrary, 
$L^{x}(\mathcal{L}^{\widetilde{\mathcal{G}}}_{1/2})$ 
is a.s. an infinite sum at a fixed point 
$x\in \widetilde{\mathcal{G}}\backslash\partial\mathcal{G}$.
However $x\mapsto 
L^{x}(\mathcal{L}^{\widetilde{\mathcal{G}}}_{1/2})$ admits a
continuous version (\cite{Lupu2016Iso}), and we will only consider it.
We will also consider the clusters formed by 
$\mathcal{L}^{\widetilde{\mathcal{G}}}_{1/2}
\cup \Xi_{u}^{\widetilde{\mathcal{G}}}$.
Two trajectories (loops or excursions) belong to the same \textit{cluster} if there is a finite chain of trajectories which connects the two, such that any two consecutive elements of the chain intersect each other.
The zero set of 
$L^{x}(\mathcal{L}^{\widetilde{\mathcal{G}}}_{1/2})+
L^{x}(\Xi_{u}^{\widetilde{\mathcal{G}}})$, which is non-empty with positive probability, is exactly the set of points not visited by any loop or excursion. The connected components of the positive set of
$L^{x}(\mathcal{L}^{\widetilde{\mathcal{G}}}_{1/2})+
L^{x}(\Xi_{u}^{\widetilde{\mathcal{G}}})$ are exactly the clusters of
$\mathcal{L}^{\widetilde{\mathcal{G}}}_{1/2}
\cup \Xi_{u}^{\widetilde{\mathcal{G}}}$, i.e. all the trajectories inside such a connected component belong to the same cluster.
In \cite{Lupu2016Iso}, it is proved only for clusters of loops, but one can easily generalize it to the case with excursions.
Also note that on the metric graph with positive probability the clusters of loops and excursions are strictly larger than the ones on the discrete network, i.e. they connect more vertices. We state next isomorphism without proof as it can be deduced from Proposition \ref{PropIsoDiscr} following the method of \cite{Lupu2016Iso}.

\begin{prop}
	\label{PropIsoCable}
	Let $u$ be a non-negative boundary condition and 
	$\mathcal{L}^{\widetilde{\mathcal{G}}}_{1/2}$ and
	$\Xi_{u}^{\widetilde{\mathcal{G}}}$ be as previously. Let $\sigma(x)$ be a random sign function with values in $\lbrace -1,1\rbrace$, defined on the set
	\begin{displaymath}
	\lbrace x\in\widetilde{\mathcal{G}}\vert 
	L^{x}(\mathcal{L}^{\widetilde{\mathcal{G}}}_{1/2})+
	L^{x}(\Xi_{u}^{\widetilde{\mathcal{G}}})>0\rbrace,
	\end{displaymath}
	such that
	\begin{itemize}
		\item $\sigma(x)$ is constant on the connected components of its domain,
		\item conditionally on 
		$(\mathcal{L}^{\widetilde{\mathcal{G}}}_{1/2},
		\Xi_{u}^{\widetilde{\mathcal{G}}})$, the value of $\sigma(x)$ is independent of the values of $\sigma$ on other connected components, 
		\item $\sigma(x)$ equals $1$ if the cluster of $x$ contains at least one excursion,
		\item if the cluster of $x$ contains no excursion (or equivalently does not intersect $\partial\mathcal{G}$), then conditionally on 
		$(\mathcal{L}^{\widetilde{\mathcal{G}}}_{1/2},
		\Xi_{u}^{\widetilde{\mathcal{G}}})$,
		$\sigma(x)$ equals $-1$ or $1$ with probability $1/2$ each.
	\end{itemize}
	The definition of $\sigma$ will be extended to $\widetilde{\mathcal{G}}$
	by letting $\sigma$ to equal $0$ on 
	$\lbrace x\in\widetilde{\mathcal{G}}\vert 
	L^{x}(\mathcal{L}^{\widetilde{\mathcal{G}}}_{1/2})+
	L^{x}(\Xi_{u}^{\widetilde{\mathcal{G}}})=0\rbrace$.
	Then the field
	\begin{displaymath}
	\left(\sigma(x)\sqrt{2}
	\left(L^{x}(\mathcal{L}^\mathcal{G}_{1/2})
	+L^{x}(\Xi_{u}^{\mathcal{G}})\right)^{1/2}\right)
	_{x\in \widetilde{\mathcal{G}}}
	\end{displaymath}
	is distributed like $\tilde{\phi}+u$, the metric graph GFF with boundary condition $u$.
\end{prop}

\subsection{First passage sets of the GFF on a metric graph}\label{dFPS}

There is a natural notion of \textit{first passage sets} for the metric graph GFF $\tilde{\phi}+u$, which are analogues of first passage bridges for the one-dimensional Brownian motion. 
Let $a\in\mathbb{R}$. Define
\begin{displaymath}
\widetilde{\A}_{-a}^{u}=
\widetilde{\A}_{-a}^{u}(\tilde{\phi}):=
\lbrace x\in\widetilde{\mathcal{G}}\vert
\exists \gamma~\text{continuous path from}~x~\text{to}~
\partial\mathcal{G}~\text{such that}~\tilde{\phi}\geq -a~\text{on}~\gamma\rbrace.
\end{displaymath}
We report to Figure \ref{FigLvlLM} for a picture of a first passage set on metric graph.
$\widetilde{\A}_{-a}^{u}$ is a compact optional set and 
$\tilde{\phi}+u$ equals $-a$ on 
$\partial \widetilde{\A}_{-a}^{u}\backslash\partial\mathcal{G}$. Moreover, each connected component of 
$\widetilde{\A}_{-a}^{u}$ intersects $\partial\mathcal{G}$. 
$\partial \widetilde{\A}_{-a}^{u}$ is the first passage set of 
level $-a$. These first passage sets were introduced in
\cite{LupuWerner2016Levy}.
From Proposition \ref{PropIsoCable} we obtain a representation of the FPS using Brownian loops and excursions:

\begin{prop}
	\label{CorFPCluster}
	If $a=0$ and the boundary condition $u$ is non-negative, then in the coupling of
	Proposition \ref{PropIsoCable}, $\widetilde{\A}_{0}^{u}$
	is the union of topological closures of clusters of loops and excursions that contain at least an excursions
	(i.e. are connected to $\partial\mathcal{G}$), plus
	$\partial\mathcal{G}$.
\end{prop}

%%%%%%%%%%%%%%%%%%%%%%%%%%%%%%%%%%%%%%%%%%%%%%%%%%%%%%%%%%%%%%%%%%%%%%%%%

\section{Continuum preliminaries}\label{sectC}

In this section, we discuss about the continuum counterpart of the objects defined in the last section. First, we recall the notion of the continuum two-dimensional GFF. Then, we discuss Brownian loop and excursion measures. Further, we will give an isomorphism relating Brownian loops and excursions to the Wick square of the GFF. Finally, we will recall some properties of the first passage set of the continuum GFF that appear in \cite{ALS1}.

We denote by $D\subseteq \mathbb{C}$ an open bounded planar domain with a non-empty and non-polar boundary. By conformal invariance, we can always assume that $D$ is a subset of the unit disk $\D$.
The most general case that we work with are domains $D$ such that the complement of $D$ has at most finitely many connected component and no complement being a singleton. Recall that by the Riemann mapping for multiple-connected domains \cite{Ko2}, such domains $D$ are known to be conformally equivalent to a circle domain (i.e. to $\D \backslash K$, where $K$ is a finite union of closed disjoint disks, disjoint also from $\partial \D$). 

\subsection{The continuum GFF and its local sets}

The (zero boundary) Gaussian Free Field (GFF) in a domain $D$ 
\cite{SGFF} can be viewed as a centered Gaussian process $\Phi$ (we also sometimes write  $\Phi^D$ when we the domain needs to be specified) 
indexed by the set of continuous functions with compact support in $D$, with covariance given by 
$$ \E [(\Phi,f_{1}) (\Phi,f_{2})]  =  
\iint_{D\times D} f_{1}(z) G_D(z,w) f_{2}(w) \d z \d w, $$ 
where $G_D$ is the Green's function of Laplacian (with Dirichlet boundary conditions) in $D$, normalized such that $G_D(z,w)\sim (2\pi)^{-1} \log(1/|z-w|)$ as $z \to w$. The corresponding Dirichlet form is
$\int_{D} (\nabla f)^{2}$.

For this choice of normalization of $G$ (and therefore of the GFF), we set  $$\lambda=\sqrt{\pi/8}.$$ 
In the literature, the constant $2\lambda$ is called the \textit{height gap} \cite{SchSh,SchSh2}. Sometimes, other normalizations are used in the literature: if ${G_D (z,w) \sim c \log(1/|z-w|)}$ as $z \to w$, then $\lambda$ should 
be taken to be $(\pi/2)\times \sqrt {c}$. 

The covariance kernel of the GFF blows up on the diagonal, which makes it impossible to view $\Phi$ as a random function defined pointwise. It can, however, be shown that the GFF has a version that lives in some space of generalized functions (Sobolev space $H^{-1}$), which justifies the notation $(\Phi,f)$ for $\Phi$ acting on functions $f$ (see for example \cite{Dub}).

In this paper, $\Phi$ always denotes the zero boundary GFF. We also consider GFF-s with non-zero Dirichlet boundary conditions - they are given by $\Phi + u$ where $u$ is some bounded harmonic function whose boundary values are piecewise constant\footnote{Here and elsewhere this means piecewise constant that changes only finitely many times}.

\subsection{Local sets: definitions and basic properties}
Let us now introduce more thoroughly the local sets of the GFF. We only discuss items that are directly used in the current paper. For a more general discussion of local sets, thin local sets (not necessarily of bounded type), we refer to 
\cite{SchSh2,WWln2,Se}.

Even though, it is not possible to make sense of $(\Phi,f)$ when $f=\1_A$ is the indicator function of an arbitrary random set $A$, local sets form a class of random sets where this is (in a sense) possible:

\begin{defn}[Local sets]
	Consider a random triple $(\Phi, A,\Phi_A)$, where $\Phi$ is a  GFF in $D$, $A$ is a random closed subset of $\overline D$ and $\Phi_A$ a random distribution that can be viewed as a harmonic function when restricted to
	$D \backslash A$.
	We say that $A$ is a local set for $\Phi$ if conditionally on $(A,\Phi_A)$, $\Phi^A:=\Phi - \Phi_A$ is a  GFF in $D \backslash A$. 
	\end {defn}
	Throughout this paper, we use the notation $h_A: D\rightarrow\R$ for the function  that is equal to $\Phi_A$ on $D\backslash A$ and $0$ on $A$.

	Let us list a few properties of local sets (see for instance \cite {SchSh2,Aru,AS} for derivations and further properties): 
	\begin{lemma}\label{BPLS}    $\ $
		\begin {enumerate}
		\item Any local set can be coupled in a unique way with a given GFF: Let $(\Phi,A,\Phi_A,\widehat \Phi_A)$ be a coupling, where $(\Phi,A,\Phi_A)$ and $(\Phi,A,\widehat \Phi_A)$ satisfy the conditions of this definition. Then, a.s. $\Phi_A=\widehat \Phi_A$. Thus, being a local set is a property of the coupling $(\Phi,A)$, as  $\Phi_A$ is a measurable function of $(\Phi,A)$. 
		\item (Proposition 1.3.29 of \cite{Aru}) If $A$ and $B$ are local sets coupled with the same GFF $\Phi$, and $(A, \Phi_A)$ and $(B, \Phi_B)$ are conditionally independent given $\Phi$, then $A \cup B$ is also a local set coupled with $\Phi$. Additionally, $B\backslash A$ is a local set of $\Phi^A$ with $(\Phi^A)_{B\backslash A} = \Phi_{B\cup A}-\Phi_{A}$. 
	\end{enumerate}
\end{lemma}
\subsection{First passage sets of the 2D continuum GFF}

The aim of this section is to recall the definition first passage sets introduced in \cite{ALS1} of the 2D continuum GFF, and state the properties that will be used in this paper.

The set-up is as follows: $D$ is a finitely-connected domain where no component is a single point and $u$ is a bounded harmonic function with piecewise constant boundary conditions.
\begin{defn}[First passage set]\label{Def ES}
	Let $a\in \R$ and $\Phi$ be a GFF in the multiple-connected domain $D$. We define the first passage set of $\Phi$ of level $-a$ and boundary condition $u$ as the local set of $\Phi$ such that $\partial D \subseteq \A^u_{-a}$, with the following properties:
	\begin{enumerate}
		\item Inside each connected component $O$ of $D\backslash \A_{-a}^u$, the harmonic function $h_{\A_{-a}^u}+u$ is equal to $-a$ on $\partial \A_{-a}^u \backslash \partial D$ and equal to $u$ on $\partial D \backslash \A_{-a}^u$ in such a way that $h_{\A_{-a}^u}+u \leq -a$. 		
		\item $\Phi_{\A^u_{-a}}-h_{\A_{-a}^u}\geq 0$, i.e., for any smooth positive test function $f$ we have 
		$(\Phi_{\A^u_{-a}}-h_{\A_{-a}^u},f) \geq 0$.
		\item For any connected component $O$ of $D\backslash \A_{-a}^u$, $\eps > 0$ and $z \in \partial O$, and for all sufficiently small open ball $U_z$ around $z$,  we have that a.s. \[h_{\A_{-a}^u}(z)+u(z) \geq \min\left \{ -a,\inf_{w \in U_z\cap \overline O} u(w)\right \} - \eps.\] 
	\end{enumerate}	
\end{defn}

Notice that if $u\geq -a$, then the conditions (1) and (2) correspond to
\begin{enumerate}
	\item [(1)'] $h_{\A_{-a}^u}+u=-a$ in $D\backslash \A_{-a}^u$.
	\item [(2)'] $\Phi_{\A_{-a}^u}+u+a\geq 0$.	
\end{enumerate}

Moreover, in this case the technical condition (3) is not necessary. This condition roughly says that nothing odd can happen at boundary values that we have not determined: those on the intersection $\partial \A_{-a}^u$ and $\partial D$. This condition enters in the case $u < -a$ as we want to take the limit of the FPS on metric graphs and it comes out that it is easier not to prescribe the value of the harmonic function at the intersection of $\partial D$ and $\partial\A_{-a}^u$.

\begin{rem}\label{bw}
	One could similarly define excursions sets in the other direction, i.e. stopping the sets from above. We denote these sets by $\AB_b$. In this case the definition goes the same way except that (2) should now be replaced by $\Phi_{\AB_{b}^u}\leq h_{\AB_b^u}$,and (3) by
	\[h_{\A_{b}^u}(z)+u(z) \leq \max\left \{ b,\inf_{w \in U_z\cap \overline O} u(w)\right \}+ \eps.\] 
\end{rem}

We now present the key result in the study of FPS.
\begin{thm}[Theorem 4.3 and Proposition 4.5 of \cite{ALS1}]\label{Thm::FPS}Let $D$ be a finitely connected domain, $\Phi$ a GFF in $D$ and $u$ be a bounded harmonic function that has piecewise constant boundary values. Then for all $a\geq 0$,  the first passage set of $\Phi$ of level -a and boundary condition $u$, $\A_{-a}^u$, exist and satisfy the following property:
	\begin{enumerate}
		\item Uniqueness: if $A'$ is another local set coupled with $\Phi$ and satisfying \ref{Def ES}, then a.s. $A'=\A_{-a}^u$.
		\item Measurability: $\A_{-a}^u$ is a measurable function of $\Phi$.
		\item Monotonicity: If $a\leq a'$ and $u\leq u'$, then $\A_{-a}^u\subseteq \A_{-a'}^{u'}$ .
	\end{enumerate}
\end{thm}

\subsection{Brownian loop and excursion measures}

Next, we discuss Brownian loop and excursion measures in the continuum. Consider a non-standard Brownian motion $(B_{t})_{t\geq 0}$ on 
$\mathbb{C}$, such that its infinitesimal generator is the Laplacian $\Delta$, so that $\E\left[\|B_t\|^2\right]=4t$ . The reason we use a non-standard Brownian motion comes from the fact that the isomorphisms with the continuum GFF have nicer forms. We will denote
$\mathbb{P}_{t}^{z,w}$ the bridge probability measures
corresponding to $(B_{t})_{t\geq 0}$. 
Given $D$ an open subset of $\mathbb{C}$, we will denote
\begin{displaymath}
T_{\partial D}:=
\inf\lbrace t\geq 0\vert B_{t}\not\in D\rbrace
\in (0,+\infty].
\end{displaymath}
The \textit{Brownian loop measure} on $D$ is defined as
\begin{displaymath}
\mu^{D}_{\rm loop}(\cdot)=
\int_{D}\int_{0}^{+\infty}
\mathbb{P}_{t}^{z,z}(\cdot, T_{\partial D}>t)\dfrac{1}{4\pi t}
\dfrac{dt}{t}dz,
\end{displaymath}
where $dx$ denotes the Lebesgue measure on $\mathbb{C}$.
This is a measure on rooted loops, but it is natural to consider unrooted loops, where one ``forgets'' the position of the start and endpoint.
This Brownian loop measure was introduced in
\cite{LW2004BMLoopSoup}, see also \cite{LawC}, 
Section 5.6 \footnote{In \cite{LawC}, 
	Section 5.6 the authors rather consider the loop measure associated to a standard Brownian motion. This is just a matter of a change of time $ds = dt/\sqrt{2}$.}
From the definition follows that the Brownian loop measure satisfies a \textit{restriction property}: if
$D'\subset D$,
\begin{displaymath}
d\mu^{D'}_{\rm loop}(\gamma)=
\1_{\gamma~\text{contained in}~D'}
d\mu^{D}_{\rm loop}(\gamma).
\end{displaymath}
It also satisfies a conformal invariance property. The image of
$\mu^{D}_{\rm loop}$ by a conformal transformation of $D$ is
$\mu^{D}_{\rm loop}$ up to a change of root and time reparametrization.
In particular, the measure on the range of the loop is conformal invariant. For $\mu^{\mathbb{C}}_{\rm loop}$, there is also invariance by polar inversions (up to change of root and reparametrization). 
A \textit{Brownian loop-soup} in $D$ with intensity parameter
$\alpha>0$ is a Poisson point process of intensity measure
$\alpha\mu^{D}_{\rm loop}$, which we will denote by
$\mathcal{L}^{D}_{\alpha}$.\\

Now we get to the excursion measure. 
$H_{D}(dx,dy)$ will denote the \textit{boundary Poisson kernel} on
$\partial D\times\partial D$.
In the case of domains with $\mathcal{C}^{1}$ boundary, the boundary Poisson kernel is defined as 
\begin{displaymath}
H_{D}(x,y)=\partial_{n_{x}}\partial_{n_{y}} G_{D}(x,y),~
x,y\in\partial D,
\end{displaymath}
For general case, we use the conformal invariance of the measure
$H_{D}(x,y) dx dy$, where $dx$ and $dy$ are arc lengths, and can define $H_{D}(dx,dy)$ as a measure. See \cite{ALS1} for details.
Given $x\neq y\in\partial D$, 
$\mathbb{P}_{\rm exc}^{D,x,y}$ will denote the probability measure on the boundary-to-boundary Brownian excursion in $D$ from $x$ to $y$, associated to the non-standard Brownian motion of generator $\Delta$.
Let $u$ be a non-negative bounded Borel-measurable function on
$\partial D$. We define the boundary-to-boundary excursion measure associated to $u$ as 
\begin{displaymath}
\mu^{D,u}_{\rm exc}=
\dfrac{1}{2}\iint_{\partial D\times\partial D}
u(x)u(y)\mathbb{P}_{\rm exc}^{D,x,y} H_{D}(dx,dy).
\end{displaymath}
These excursion measure are analogous to the one on metric graphs defined in Section \ref{Def Loop Met}.
In the particular case of $D$ simply connected and $u$ positive constant on a boundary arc and zero elsewhere, the measure 
$\mu^{D,u}_{\rm exc}$ appears in the construction of restriction measures (\cite{LawlerSchrammWerner2003ConformalRestr} and \cite{Werner2005ConfRestr}, Section 4.3). Next, we state without proof some fundamental properties of these excursion measures that follow just from properties of boundary Poisson kernel and 2D Brownian motion.
\begin{prop}
	\label{PropExcMeas}
	Let $D$ be a domain as above and $u$ a bounded non-negative condition. The boundary-to-boundary excursion measure $\mu^{D,u}_{\rm exc}$ satisfies the following properties:
	\begin{enumerate}
		\item Conformal invariance: [Proposition 5.27 of \cite{LawC}] Let $D'$ be a domain conformally equivalent to $D$ and $f$ a conformal transformation from $D$ to $D'$. Then
		$\mu^{D',u}_{\rm exc}$ is the image of $\mu^{D,u}_{\rm exc}$ by $f$, up to a change of time $ds=\vert f'(\gamma(t))\vert^{-2} dt$.
		\item Markov property: Let $\B$ be a compact subsets of $\partial D$ and assume that $u$ is supported on $\B$. Let $K$ be a compact subset of $\overline{D}$, at positive distance from  $\mathcal{B}$. We assume that $K$ has finitely many connected components.
		For any $F$ bounded measurable functional on paths, we have
		\begin{displaymath}
		\mu^{D,u}_{\rm exc}(F(\gamma),
		\gamma~\text{visits}~K)=
		\int \1_{\gamma_{1}(0)\in\B}
		\mathbb{E}_{\gamma_{1}(t_{\gamma_{1}})}
		\left[u(B_{T_{\partial D}})F(\gamma_{1}\circ(B_{t})
		_{0\leq t\leq T_{\partial D}})\right]
		d\mu^{D\backslash K,
			\1_{K}+u \1_{\B}}_{\rm exc}(\gamma_{1}),
		\end{displaymath}
		where $\gamma_{1}(t_{\gamma_{1}})$ is the endpoint of the path $\gamma_{1}$ and
		$\circ$ denotes the concatenation of paths.
	\end{enumerate}
\end{prop}

The Markov property above is analogous to the Markov property on metric graphs given by Proposition \ref{PropMarkovExcMG}. 

Given $\mathcal{B}_{1}$ and $\mathcal{B}_{2}$ two disjoint compact subsets of $\partial D$, we will denote 
$\M(\mathcal{B}_{1},\mathcal{B}_{2})$ the conformal modulus, which is the inverse of extremal length \cite{Ahlfors2010ConfInv, ALS1}. If 
$\mathcal{B}_{1}$ and $\mathcal{B}_{2}$ form a partition of connected components of $\partial D$, then
\begin{displaymath}
\M(\mathcal{B}_{1},\mathcal{B}_{2})=
\iint_{\mathcal{B}_{1}\times\mathcal{B}_{2}}
H_{D}(dx,dy)
=\mu^{D,1}_{\rm exc}
(\text{excursions having one end in}~\mathcal{B}_{1}
~\text{and the other in}~\mathcal{B}_{2}).
\end{displaymath} 
In general,
\begin{displaymath}
M(\B_{1},\B_{2})\geq
\iint_{\B_{1}\times\B_{2}}
H_{D}(dx,dy)
.
\end{displaymath}

\subsection{Wick square of the continuum GFF and isomorphisms}

The isomorphism theorems on discrete or metric graph (Propositions
\ref{PropIsoDiscr}, \ref{PropIsoCable}) involve the square of a GFF. However for the continuum GFF in dimension 2, which is a generalized function, the square is not defined. Instead, one can define a renormalized square, the \textit{Wick square} 
(\cite{Simon1974EQFT, Janson1997GaussHilbSpaces}).

Let $D$ be as in the previous subsection an open connected bounded domain, delimited by finitely many simple curves. First, we consider
$\Phi$ the GFF with $0$ boundary condition. 
$\Phi_{\varepsilon}$ will denote regularizations of $\Phi$ by convolution with a kernel. The Wick square $:\Phi^{2}:$ is the limit as
$\varepsilon\to 0$ of
\begin{displaymath}
\Phi^{2}_{\varepsilon}-\mathbb{E}[\Phi^{2}_{\varepsilon}].
\end{displaymath}
If $f$ is a continuous bounded test function,
\begin{displaymath}
(\Phi^{2}_{\varepsilon},f)-
\mathbb{E}[(\Phi^{2}_{\varepsilon},f)]
\end{displaymath}
converges in $\mathbb{L}^{2}$, and at the limit,
\begin{displaymath}
\mathbb{E}[(:\Phi^{2}:,f)^{2}]=
\iint_{D\times D} f(z) 2G_{D}(z,w)^{2} f(w)
dz dw.
\end{displaymath}
$:\Phi^{2}:$ is a random generalized function, measurable with respect to
$\Phi$, which lives in the Sobolev space $H^{-1}(D)$, that is to say in the completion of the space of continuous compactly supported functions in $D$ for the norm
\begin{displaymath}
\Vert f\Vert_{H^{-1}}^{2}=
\iint_{D\times D} f(z) G_{D}(z,w) f(w)
dz dw.
\end{displaymath}
See \cite{Dub}, Section 4.2. Indeed,
\begin{displaymath}
\mathbb{E}[\Vert :\Phi^{2}:\Vert_{H^{-1}}^{2}]=
\iint_{D\times D} f(z) 2G_{D}(z,w)^{3} f(w)
dz dw<+\infty.
\end{displaymath}

In \cite{LeJan2010LoopsRenorm, LeJan2011Loops}, Le Jan considers 
(following \cite{LW2004BMLoopSoup, LawlerSchrammWerner2003ConformalRestr, SheffieldWerner2012CLE}) Brownian loop-soups in $D$, $\mathcal{L}^{D}_{\alpha}$, which are Poisson point processes with intensity 
$\alpha\mu^{D}_{\rm loop}$. One sees $\mathcal{L}^{D}_{\alpha}$ as a random countable collection of Brownian loops. He considers the 
\textit{centred occupation field} of
$\mathcal{L}^{D}_{\alpha}$. The occupation field of
$\mathcal{L}^{D}_{\alpha}$ with ultra-violet cut-off is
\begin{displaymath}
(L_{\varepsilon}(\mathcal{L}^{D}_{\alpha}),f)=
\sum_{\substack{\gamma\in \mathcal{L}^{D}_{\alpha}\\
		T_{\gamma}>\varepsilon}}
\int_{0}^{T_{\gamma}} f(\gamma(t)) dt,
\end{displaymath}
where $f$ is a test function and
$T_{\gamma}$ is the life-time of a loop $\gamma$. The measure
$L_{\varepsilon}(\mathcal{L}^{D}_{\alpha})$ diverges as
$\varepsilon\to 0$, i.e. in the limit we get something which is even not locally finite. The centred occupation field is
\begin{displaymath}
L_{\rm ctr}(\mathcal{L}^{D}_{\alpha})=
\lim_{\varepsilon\to 0} L_{\varepsilon}(\mathcal{L}^{D}_{\alpha})-
\mathbb{E}[L_{\varepsilon}(\mathcal{L}^{D}_{\alpha})].
\end{displaymath}
The convergence above, evaluated against a bounded test function, is
in $\mathbb{L}^{2}$. For $\alpha=1/2$, Le Jan shows the following isomorphism:

\begin{prop}[Renormalized Le Jan's isomorphism, 
	\cite{LeJan2010LoopsRenorm, LeJan2011Loops}]
	\label{PropIsoLeJanRenorm}
	The centred occupation field
	$L_{\rm ctr}(\mathcal{L}^{D}_{1/2})$ has the same law as
	half the Wick square,
	$\frac{1}{2}:\Phi^{2}:$, where
	$\Phi$ is the GFF in $D$ with zero boundary condition.
\end{prop}

We consider now a bounded non-negative boundary condition $u$ and also denote by $u$ its harmonic extension to $D$.
Consider the GFF with boundary condition $u$,
$\Phi + u$. One can define its Wick square as
\begin{displaymath}
:(\Phi+u)^{2}:=
:\Phi^{2}:+ 2 u\Phi.
\end{displaymath}

Let $\Xi^{D}_{u}$ be a Poisson point process of boundary-to-boundary excursions with intensity $\mu^{D, u}_{\rm exc}$. The occupation
field $L(\Xi^{D}_{u})$ is well defined an it is a measure. One can still introduce the centered occupation field as
\begin{displaymath}
L_{\rm ctr}(\Xi^{D}_{u})=
L(\Xi^{D}_{u})-\mathbb{E}[L(\Xi^{D}_{u})].
\end{displaymath}
Below we extend the renormalized Le Jan's isomorphism to the case of non-negative boundary conditions. For an analogous statement in dimension 3 see \cite{Sznitam2013BMI3D}.

\begin{prop}
	\label{PropIsoURenorm}
	Let $u: \partial D\rightarrow\mathbb{R}_{+}$ be a bounded piecewise continuous (with left and right limits ant discontinuity points) non-negative boundary condition. Consider two independent Poisson point processes
	$\mathcal{L}^{D}_{1/2}$ and $\Xi^{D}_{u}$ of loops and boundary-to-boundary excursions. The field
	\begin{displaymath}
	L_{\rm ctr}(\mathcal{L}^{D}_{1/2})+
	L(\Xi^{D}_{u})
	\end{displaymath}
	has same law as
	\begin{displaymath}
	\frac{1}{2}:\Phi^{2}:+ u\Phi+
	\frac{1}{2} u^{2}.
	\end{displaymath}
	In particular, the field
	$L_{\rm ctr}(\mathcal{L}^{D}_{1/2})+L_{\rm ctr}(\Xi^{D}_{u})$
	has same law as $\frac{1}{2}:(\Phi+u)^{2}:$.
\end{prop}

\begin{proof}
	We need to show that for every non-negative continuous compactly supported function $\chi$ on $D$,
	\begin{displaymath}
	\mathbb{E}\left [e^{-(L_{\rm ctr}(\mathcal{L}^{D}_{1/2})+L(\Xi^{D}_{u}),\chi)}\right ]=
	\mathbb{E}\left [e^{-\frac{1}{2}(:\Phi^{2}:,\chi)}
	e^{-(\Phi,u\chi)}\right ]
	e^{-\frac{1}{2}(u^{2},\chi)}.
	\end{displaymath}
	By Le Jan's isomorphism (Proposition \ref{PropIsoLeJanRenorm}), we know that
	\begin{displaymath}
	\mathbb{E}\left [e^{-(L_{\rm ctr}(\mathcal{L}^{D}_{1/2}),\chi)}\right ]=
	\mathbb{E}\left [e^{-\frac{1}{2}(:\Phi^{2}:,\chi)}\right ].
	\end{displaymath}
	For the finiteness and the expression of 
	$\mathbb{E}\left [e^{-\frac{1}{2}(:\Phi^{2}:,\chi)}\right ]$, 
	see Sections 10.1 and 10.2 in \cite{LeJan2011Loops}.
	Since $\Xi^{D}_{u}$ is independent from $\mathcal{L}^{D}_{1/2}$, 
	we need to show that
	\begin{displaymath}
	\mathbb{E}\left [e^{-(L(\Xi^{D}_{u}),\chi)}\right ]=
	e^{-\frac{1}{2}(u^{2},\chi)}
	\dfrac{\mathbb{E}\left [e^{-\frac{1}{2}(:\Phi^{2}:,\chi)}
		e^{-(\Phi,u\chi)}\right ]}
	{\mathbb{E}\left [e^{-\frac{1}{2}(:\Phi^{2}:,\chi)}\right ]}.
	\end{displaymath}
	We will use the following lemma, whose proof we postpone:
	
	\begin{lemma}
		\label{LemMassive}
		The massless GFF $\Phi$, weighted by (change of measure)
		\begin{displaymath}
		\dfrac{e^{-\frac{1}{2}(:\Phi^{2}:,\chi)}}
		{\mathbb{E}[e^{-\frac{1}{2}(:\Phi^{2}:,\chi)}]},
		\end{displaymath}
		has the law of a massive GFF with 0 boundary conditions, corresponding to the Dirichlet form
		\begin{displaymath}
		\int_{D}(\nabla f)^{2}+\int_{D}\chi f^{2}.
		\end{displaymath}
	\end{lemma}
	
	From above lemma follows that
	\begin{displaymath}
	\dfrac{\mathbb{E}\left [e^{-\frac{1}{2}(:\Phi^{2}:,\chi)}
		e^{-(\Phi,u\chi)}\right ]}
	{\mathbb{E}\left [e^{-\frac{1}{2}(:\Phi^{2}:,\chi)}\right ]}=
	\exp\left(\dfrac{1}{2}\iint_{D\times D}
	u(z)\chi(z)G_{\chi}(z,w)u(w)\chi(w) dz dw\right),
	\end{displaymath}
	where $G_{\chi}$ is the Green's function of
	$-\Delta + \chi$.
	Indeed, it is an exponential moment of a massive GFF. Thus, we have to show that
	\begin{displaymath}
	\mathbb{E}\left [e^{-(L(\Xi^{D}_{u}),\chi)}\right ]=
	\exp\left(\dfrac{1}{2}\iint_{D\times D}
	u(z)\chi(z)G_{\chi}(z,w)u(w)\chi(w) dz dw
	-\dfrac{1}{2}\int_{D}u(z)^{2}\chi(z)dz\right).
	\end{displaymath}
	The above relation holds at discrete level, on a lattice approximation
	of domain $D$. It is a consequence of the isomorphism of Proposition
	\ref{PropIsoDiscr}. On a continuum domain one gets this relation by convergence of excursion measures (Lemma \ref{Basic convergences 2}) and massive Green's functions (see Remark \ref{Rem:: masive}).
\end{proof}

\begin{proofMassive}
	It is enough to show that for every constant $\varepsilon>0$,
	\begin{displaymath}
	\dfrac{e^{-\frac{1}{2}(:\Phi^{2}:,\varepsilon+\chi)}}
	{\mathbb{E}[e^{-\frac{1}{2}(:\Phi^{2}:,\varepsilon+\chi)}]}
	\end{displaymath}
	is the density of a massive GFF corresponding to the Dirichlet form
	\begin{displaymath}
	\int_{D}(\nabla f)^{2}+\int_{D}(\varepsilon+\chi)f^{2}.
	\end{displaymath}
	Then, by letting $\varepsilon$ tend to 0, we get our lemma.
	For $\varepsilon>0$ fixed, we can follow step by step
	the proof of the very similar Lemma 3.7 in \cite{LacoinRhodesVargas2014Hyperb} and use as there the decomposition of
	$\Phi^{(0)}$ according the eigenfunctions (with 0 boundary condition) of the Laplace-Beltrami operator associated to the metric
	$(\varepsilon + \chi(z))^{\frac{1}{2}}\vert dz\vert$ 
	(the area element being $(\varepsilon + \chi(x)) dx$).
\end{proofMassive}

Let us note that as for the isomorphism \ref{PropIsoDiscr},
the coupling
$(2L_{\rm ctr}(\mathcal{L}^{D}_{1/2}), 
2L_{\rm ctr}(\mathcal{L}^{D}_{1/2})+2L(\Xi^{D}_{u}))$ is not the same as
$(:\Phi^{2}:, 
:\Phi^{2}:+2 u\Phi+
u^{2})$.

%%%%%%%%%%%%%%%%%%%%%%%%%%%%%%%%%%%%%%%%%%%%%%%%%%%

\section{Convergence of FPS and clusters}
\label{SecConv}

In this section, we show that the metric graph FPS converges to the continuum FPS topology. We also prove that the clusters of metric graph Brownian loop soups and boundary-to-boundary excursions converge to their continuum counterpart. Both results are about convergence in probability, with respect to Hausdorff topology on closed subsets. We start the section with detailing the set-up and recalling some basic convergence results. Thereafter, we prove the convergence of the metric graph FPS towards its continuum counterpart.

\subsection{Set-up and basic convergence results}
In this section, we set up the framework for our convergence statements. We also review some convergence results for random closed sets, random fields and path measures. Most of the content is standard, but slightly reworded and reinterpreted. For simplicity, we restrict ourselves to $\widetilde \Z^{2}_n$, the metric graph induced by vertices $(2^{-n}\Z)^{2}$ and with unit conductances on every edge. However, one should be able to extend all the convergence results to isoradial graphs without too much effort. We always consider our metric graph $\widetilde \Z^{2}_n$ to be naturally embedded in $\C$, and when we mention distances and diameters for sets living on metric graphs, we always mean the Euclidean distance inherited from $\C$.

\subsubsection{Topologies and convergences on sets and functions}

We mostly work with finitely-connected bounded domains $D$. For us, a domain is by definition open and connected. We approximate these domains with metric graph domains obtained as intersections of $\widetilde \Z^{2}_n$ with domains of $\C$, i.e. by $\widetilde D_n := \widetilde \Z^{2}_n \cap D_n$, where $D_n \to D$ in an appropriate sense detailed below. We say that such an approximation $\widetilde D_n$ satisfies the condition \hypertarget{logof}{$\logof$} if
\begin{enumerate}
	\item[$\logof$]There exists $C, C' > 0$ such that 
	$D_n \subseteq [-C,C]^2$, and the amounts of connected components of $\C\backslash D_n$ is less or equal to $C'$.  
\end{enumerate} 

At times, we also need to work in the setting where both $D$ and $D_n$ are non-connected open sets (e.g. the complement of a CLE$_4$ carpet). The same condition makes sense in this case too.

We use the following topologies for open and closed sets:
\begin{itemize}
	\item For domains $D^z$ with a marked point $z \in D^z$, approximated by marked open domains $(D_n, z_n)$ we say that $(D_n, z_n)$ converges to $(D^z,z)$ in the sense of Carathéodory if 
	\begin{enumerate}
		\item $z_n \to z$,
		\item $D^z\subseteq \cup_{N\in\N}\cap_{n\geq N} D_n$,
		\item for any $x \in \partial D$ there are $x_n \in \partial D_n$ with $x_n \to x$. 
	\end{enumerate} Notice that in this wording we have not assumed simply-connectedness, as the Carathéodory topology generalizes nicely to multiply-connected setting (e.g. see \cite{Co}). 
	\item For closed sets, we work with the Hausdorff distance: the distance between two closed sets $A, B$ is the infimum over $\eps > 0$ such that $A \subset B + B(0,\eps)$ and $B \subset A + B(0,\eps)$, where $B(0,\eps)$ is the unit disk of radius $\eps$ and we consider the Minkowski sum. It is known that the set of closed subsets of $[-C,C]^2$ is compact for the Hausdorff topology. 
	\item For open sets $D$ that may not be connected, it is convenient to consider the Hausdorff distance on their complements with respect to $[-C,C]^2$, i.e. Hausdorff distance for $[-C,C]^2\backslash D$. Notice that if $(D^z,z)$ is any pointed connected component of $D$, then convergence of $D_n \to D$ in the sense that their complements converge, implies the Carathéodory convergence of $(D_n,z_n) \to (D^z,z)$ for any $z_n \to z$; see for example Theorem 1 of \cite{Co}.
\end{itemize}

We are also interested in the convergence of functions on 
$\widetilde D_n = \widetilde \Z^2_n \cap D_n$ 
to (generalized) functions on $D \subset [-C,C]^2$. In fact, it is more convenient to look at functions, whose domain of definitions is extended to the whole of $[-C,C]^2$. Thus, we extend a function $\tilde f$ defined on $\widetilde D_n$ to the whole of $[-C,C]^2$ by taking the harmonic extension of $\tilde f$, with zero boundary values on $\partial [-C,C]^2$. In particular, this extended function $\widehat f$ is then well-defined inside the square faces delimited by $\widetilde D_n$. 

Observe that in the case of the metric graph GFF $\tilde \phi$, such an extension $\widehat \phi$ is still a Gaussian process. We use these extensions everywhere when talking about the convergences of functions and often omit the word `extension' for readability. If we want to be explicit, we use the decoration $\widehat{\ }$ as above. In particular $\widehat G_{\widetilde D_n}$ will denote the Green's function of the metric graph GFF defined on $\widetilde D_n$ and extended to $[-C,C]^2$.

Both harmonic functions and GFF-s can be considered on any open set. If $\Phi$ is a GFF in $D$, then we can write $\Phi=\sum_{D^z} \Phi^{D^z}$ where the sum runs on the connected components $D^z$ of $D$ and where $\Phi^{D^z}$ is a GFF in $D^z$ independent of all the others. We consider the following topologies for the spaces of functions:

\begin{itemize}
	\item For the convergence of the extensions of bounded functions we use the uniform norm on compact subsets of $[-C,C]^2 \backslash \partial D$. We avoid $\partial D$ because we want to allow for a finite number of jumps on $\partial D$.
	\item The GFF-s on metric graphs and on domains are always considered as elements of the Sobolev space $H^{-1-\epsilon}([-C,C]^2)$. For background on Sobolev spaces we refer the reader to \cite{Adams}.
\end{itemize}

We will shortly see that these convergences are well-behaved in the sense that natural approximations of continuum objects converge. A key ingredient is the weak Beurling estimate (see for e.g. Proposition 2.11 of \cite{ChSm} for the discrete case and Proposition 3.73 of \cite{LawC} for the continuum case):

\begin{lemma}[Beurling estimate]\label{Brl} There exists $\beta>0$ such that for all $K\subseteq \widetilde \Z^{2}_n$  with $C$ connected components all of them with size at least $\delta$, and for all $z\in \widetilde \Z^{2}_n\backslash K$ and $\epsilon \leq \delta/2$
	\begin{align*}
	\P^x(\widetilde X \text{ hits $B(z,\epsilon)$ before hitting } K)\leq  \text{const}(C,\delta)\left(\frac{d(z,K)}{\epsilon}\right)^\beta,
	\end{align*}
	where $\widetilde X$ is a metric graph Brownian motion started at $z$. The same estimate holds in the continuum, i.e. if we replace 
	$\widetilde\Z^{2}_n$ by $\C$ and consider the two-dimensional Brownian motion.
\end{lemma}

The following lemma is basically contained in \cite{ChSm} Proposition 3.3 and Corollary 3.11. Although the statements there include more stringent conditions, in particular, the boundaries are assumed to be Jordan curves and domains simply-connected, one can verify that this is not really used in the proofs. For similar statements one can also see Proposition 3.5 and Lemma A.1 in \cite{BL}.

\begin{lemma}\label{Basic convergences}
	Suppose $D$ and $(D_n)_{n\in \N}$ are open sets that satisfy condition 
	\hyperlink{logof}{$\logof$}, and $D_n\to D$ in the sense that their complements converge in the Hausdorff topology. Then, we have the following convergences:
	\begin{enumerate}
		\item  Let $H$ be a bounded function on $[-C,C]^2$ with at most a finite number of discontinuity points on $\partial D$ and continuous elsewhere, and let $u$ be the unique harmonic function on $[-C,C]^2 \backslash \partial D$ that takes the values of $H$ on $\partial D$ and the value $0$ on $\partial [-C,C]^2$. Let $\widehat u_n$ be the extension of the metric graph harmonic function defined on $\widetilde D_n$ by the restriction of $H$ to $\partial \widetilde D_n$. Then $\widehat u_n$ converge to $u$ in the sense above.
		\item For any continuous bounded $f$ defined on $Q_C = [-C,C]^2$ we have that, 
		$$\lim_{n\to +\infty}
		\iint_{Q_C \times Q_c}f(z)\widehat G_{\widetilde D_n}(z,w)
		f(w)dz dw 
		= \iint_{Q_C \times Q_C} f(z)G_{D}(z,w)f(w)dz dw,$$ 
		where $\widehat G_{\widetilde D_n}$ is the harmonic extension of the metric graph Green's function on $\widetilde D_n$.
	\end{enumerate}
	Similarly, for any connected component $D^z$ of $D$ containing $z$, if $(D_n, z_n)$ converge towards $(D,z)$ in the Carathéodory sense, then the statements also hold.
\end{lemma}

\begin{rem}
	We include the possibility of finitely many discontinuity points on $\partial D$, as then the statement provides an explicit way of constructing (metric graph) harmonic functions, whose extensions converge to the original harmonic function in the topology defined above. 
\end{rem}

\begin{proof}
	As mentioned just before the statements, the proofs are basically contained in \cite{ChSm}. Hence we will only sketch the steps with appropriate references. 
	\begin{enumerate}
		\item Pre-compactness in the uniform norm on compacts of $D$, and harmonicity outside of $\partial D$ both follow from the proof of Proposition 3.1 in \cite{ChSm}. In particular, we know that each subsequential limit is a bounded harmonic function. To determine the boundary values one uses Beurling estimate as in the proof of Proposition 3.3 in \cite{ChSm}.
		\item The convergence of the (extension of the) discrete Green's function on $[-C,C]^2 \cap \Z_n$ to the continuum Green's function on $[-C,C]^2$ is well-known and can be explicitly shown, for example, via an eigenfunction expansion of the Green's function. The convergence of $G_{[-C,C]^2 \cap \widetilde \Z_n}$ and of $\widehat G_{[-C,C]^2 \cap \widetilde \Z_n}$ then follows.
		
		For the general case, note that the function $\widehat G_{\widetilde D_n}(z,\cdot)- \widehat G_{[-C,C]^2\cap \widetilde \Z^2_n}(z,\cdot)$ is the harmonic extension of function on $\widetilde D_n$ with uniformly bounded boundary values. Thus it converges by (1)
		to $G_{D}(z,\cdot)-G_{[-C,C]^{2}}(z,\cdot)$.
		To deduce the convergence of the integral one finally uses and dominated convergence together with the fact that $\widehat G_{[-C,C]^2\cap \widetilde \Z^2_n}(z,w)$ is upper bounded by $c(\log(|z-w|)+1)$. For more details see e.g. Proposition 3.5 of \cite{BL}.
	\end{enumerate}
\end{proof}

\begin{rem}\label{Rem:: masive}
	Note that statement (2) can be proved similarly for a massive Green's function. One just need to replace the harmonic extensions by the solutions of the appropriate Poisson equation, and the standard Brownian motion by a Brownian motion killed at an appropriate exponential rate.
\end{rem}

Lemma \ref{Basic convergences} allows us to give a short argument for the convergence of the metric GFF-s:

\begin{cor}\label{GFFconv}
	Suppose $D$ and $(D_n)_{n\in \N}$ are open sets that satisfy condition \hyperlink{logof}{$\logof$}, and that $D_n\to D$ in the sense that their complements converge in the Hausdorff topology. Then the extensions $\widehat \phi_n$ of the metric graph GFF-s $\widetilde \phi_n$ on $\widetilde D_n$ converge in law in $H^{-1-\eps}([-C,C]^2)$ to a GFF $\Phi$ on $D$. Moreover, for any connected component $D^z$ of $D$ containing $z$, if $(D_n, z_n)$ converge towards $(D,z)$ in the Carathéodory sense, then the restrictions of $\widehat \phi_n$ to $D$ converge to a zero boundary GFF $\Phi^D$ on $D$.
\end{cor}

\begin{proof}
	Lemma \ref{Basic convergences} (2) guarantees the convergence of finite-dimensional marginals. Thus it remains to prove tightness.
	The norm of the Sobolev space $H^{-1}(D)$ is given by (e.g. see \cite{Dub}, Section 4.2.)
	\begin{displaymath}
	\Vert f\Vert_{H^{-1}}^{2}=
	\iint_{D\times D} f(z) G_{D}(z,w) f(w)
	dz dw.
	\end{displaymath}
	But using Lemma \ref{Basic convergences} (2) and denoting $Q_C = [-C,C]^2$ we can explicitly calculate to see that $$\sup_{n\in \N}\E\left[\|\tilde  \phi_n\|_{H^{-1}(Q_C)}^2 \right] = \sup_{n\in \N} \iint_{Q_C\times Q_C} \widehat G_{\widetilde D_n} (z,w) G_{D}(z,w)
	dz dw <\infty. $$ Hence by the Sobolev embedding theorem, we have that $(\tilde \phi_n)_{n\in \N}$ is tight in $H^{-1-\epsilon}([-C,C]^2)$ for any $\eps > 0$ and the convergence follows. 
	
	The latter part follows similarly.
\end{proof}

\subsubsection{Topologies and convergences on loops and excursions}

Now, let $\L_{\alpha}^{D}$ and $\L_{\alpha}^{\widetilde D_n}$ be respectively a continuum and a metric graph loop-soup, i.e. PPPs with intensity measures $\alpha \mu_{\rm loop}^{D}$ and $\alpha \mu_{\rm loop}^{\widetilde{D}_n}$ respectively. Moreover, for $u$ a positive function on $\partial D$ and $u_n$ a positive function on $\partial \widetilde D_n$, let $\Xi^{D}_{u}$ and $\Xi^{\widetilde D_n}_{ u_n}$ be respectively independent PPP of boundary-to-boundary Brownian excursions of intensity $\mu^{D,u}_{\rm exc}$ and boundary-to-boundary metric graph excursions of intensity $\mu^{\widetilde D_n,u_n}_{\rm exc}$. We use the following topologies when we work with paths, i.e. excursions and loops, and sets of paths:
\begin{itemize}
	\item We consider paths as closed subsets in $\overline D$ and consider the Hausdorff distance $d_H$ on these subsets.
	\item For a set of paths $\Gamma$, define $\Gamma^\eps$ as the subset $\Gamma$, consisting of paths that have diameter larger than $\epsilon$. Now on the sets $\Gamma$, for which the cardinality of $\Gamma^\eps$ is finite for all $\eps > 0$, we define the distance $d(\Gamma_1, \Gamma_2)$ to be equal to 
	$$\inf\left\{\delta > 0: \text{There is a bijection $f: \Gamma_1^\delta \to \Gamma_2^\delta$ with $\sup_{\ell  \in \Gamma_1^\delta}d_H(\ell, f(\ell)) \leq \delta$}\right\}.$$
	One can verify that $d(\Gamma_n, \Gamma) \to 0$ is equivalent to the existence of $\delta_k \to 0$ such that $\Gamma_n^{\delta_k} \to \Gamma^{\delta_k}$ in the sense that there exists a sequence of bijections $f_n : \Gamma_n^{\delta_k} \to \Gamma^{\delta_k}$ such that $\sup_{\ell \in  \Gamma_n^{\delta_k}} d_H(\ell, f(\ell)) \to 0$, and that moreover the sets of path for which the cardinality of $\Gamma^\epsilon$ is finite for all $\eps > 0$, endowed with this distance, defines a Polish space.
\end{itemize}

The following lemma says that these convergences also behave nicely:

\begin{lemma}\label{Basic convergences 2}
	Suppose $D$ and $(D_n)_{n\in \N}$ are open sets that satisfy condition \hyperlink{logof}{$\logof$}, and $D_n\to D$ in the sense that their complements converge in the Hausdorff topology. Moreover, let $u$ be a positive harmonic function in $[-C,C]^2 \backslash \partial D$ defined by piecewise constant boundary values on $\partial D$ and $u_n$ harmonic functions on $\widetilde D_n$ converging to $u$. Then, we have that for all $\epsilon \geq 0$:
	\begin{enumerate}
		\item  $\mu_{\rm loop}^{\widetilde{D}_n}\1_{\operatorname{Diam}(\gamma)\geq \epsilon}\to \mu_{\rm loop}^{D}\1_{\operatorname{Diam}(\gamma)\geq \epsilon}$ weakly w.r.t the Hausdorff distance. It follows that for all $\alpha \geq 0$ we can define on the same probability space $\L_{\alpha}^{D}$, and $(\L_{\alpha}^{\widetilde D_n})_{n \in \N}$ so that $\L_{\alpha}^{\widetilde D_n}\to \L_{\alpha}^{D}$ almost surely, w.r.t. the topology on the sets of paths defined above. 
		\item $\mu^{\widetilde D_n,u_n}_{\rm exc}
		\1_{\operatorname{Diam}(\gamma)\geq \epsilon}\to \mu^{\widetilde D,u}_{\rm exc}\1_{\operatorname{Diam}(\gamma)\geq \epsilon}$ weakly w.r.t the Hausdorff distance. It follows that for all $\alpha \geq 0$ we can define on the same probability space $(\Xi_{u_n}^{\widetilde D_n})_{n \in \N}$ and $\Xi_{u}^{D}$ such that $\Xi_{u_n}^{\widetilde D_n}\to \Xi_{u}^{D}$ almost surely, w.r.t. the topology on sets of paths defined above.
	\end{enumerate}
	Similarly, for any connected component $D^z$ of $D$ containing $z$, if $(D_n, z_n)$ converge towards $(D,z)$ in the Carathéodory sense, then the statements also hold. 
\end{lemma}

\begin{proof}
	In both points (1) and (2) the second conclusion follows directly from the first. For example, in the case (1) we can choose $\delta_k \to 0$ such that the PPPs of intensity measures $\mu_{\rm loop}^{\widetilde{D}_n}\1_{\operatorname{Diam}(\gamma)\geq \delta_k}$ converge jointly in law to PPPs of intensity measure $\mu_{\rm loop}^{D}\1_{\operatorname{Diam}(\gamma)\geq \delta_k}$. By Skorokhod representation theorem we can couple them all on the same probability space to have an almost sure convergence of these PPPs. But then by the equivalent description of the topology on sets of paths given above, we obtain the second conclusion. Thus, in what follows we just prove the first statement for both (1) and (2).	
	
	(1) The statement for random walk loop-soups on $\Z^2_n \cap D^z$ for a domain $D^z$ follows from Corollary 5.4 of \cite{LawlerFerreras2007RWLoopSoup}. The proof for the metric graph 
	loop-soups in that context is exactly the same. As remarked just after the proof (of Corollary 5.4 of \cite{LawlerFerreras2007RWLoopSoup}), the ideas extend to our non-simply connected case with finitely many boundary components. Moreover, one can verify that one can also approximate $D^z$ using $\Z^2_n \cap D_n$ where $(D_n,z_n) \to (D^z,z)$ in the sense of Carathéodory. As the convergence of $D_n \to D$ in the sense that the complements converge in the Hausdorff metric implies the Carathéodory convergence for all components, and we have only countably many components, the claim follows. 
	
	(2) Essentially the proof follows the steps of \cite{LawlerFerreras2007RWLoopSoup}: we need to first show convergence of excursions with diameter larger than $\eps$ that visit some compact set inside $D$, and then to show that there are no excursions of diameter $\eps$ that stay $\delta$ close to the boundary. 
	
	For the first part it suffices to show that for any closed square $Q\subseteq D$ with rational endpoints, we have weak convergence $\1_{\gamma\cap Q\neq\emptyset}\mu_{\rm exc}^{\widetilde D_n,u_{n}}\to \1_{\gamma\cap Q\neq\emptyset}\mu_{\rm exc}^{D,u}$. This follows from the Markov property for the metric graph excursions (Proposition \ref{PropMarkovExcMG}) and the Brownian excursion measure (Proposition \ref{PropExcMeas}). Indeed, we can decompose the excursions in $D$
	(or $\widetilde D_n$) at their first hitting time at $Q$ into an excursion from 
	$\partial D$ (or 
	$\partial\widetilde D_n$) to
	$\partial Q$ and a Brownian motion (continuum 2D or on metric graph) started on $\partial Q$ and stopped at its first hitting time of 
	$\partial D$ (or $\partial \widetilde D_n$). The convergence of the second part just follows from the convergence of random walks to Brownian motion inside compacts of $D$ and Beurling estimate for the convergence of the actual hitting point. For the excursion from 
	$\partial D$ (or $\partial\widetilde D_n$) to
	$\partial Q$, we can decompose it further into an excursion from 
	$\partial Q'$ to $\partial Q$, where $Q'$ is some closed square with rational endpoints containing $Q$ in its interior, and a time-reversed Brownian motion (continuum 2D or on metric graph) from $\partial Q'$ to the boundary of $\partial D$ (or $\partial \widetilde D_n$). The convergence of both pieces is now clear. 
	
	Finally, we need to show that for all $\epsilon>0$, 
	$$\lim_{\delta\to 0}\limsup_{n\to +\infty} \mu_{\rm exc}^{\widetilde D_n,u_{n}}\left (\operatorname{Diam}(\gamma)\geq 2\epsilon, \sup_{x\in \gamma} d(x,\partial D)\leq \delta\right )= 0.$$ 
	To do this we can again use the Markov decomposition. We cover the boundary of $\widetilde D_{n}$, for all $n$, with open disks $(B(z_{i},\eps))_{i \in I}$. The minimal number of disks needed depends on $\eps$, but is uniformly upper bounded in $n$. Any excursion that is at least $2\eps$ in diameter and has one endpoint in $B(z_{i},\eps)$, has to hit $\partial B(z_{i},\eps)$. But then it can be decomposed into an excursion from $\partial \widetilde D_{n}$ to 
	$\partial B(z_{i},\eps)$ and a metric graph BM from $\partial B(z_{i},\eps)$ to $\partial \widetilde D_{n}$. The probability that the latter goes $\eps$ far without getting $\delta$ far from $\partial D$ can be bounded by Beurling estimate (Lemma \ref{Brl}) and goes to 0 as 
	$\delta \to 0$ uniformly in sufficiently large $n$.
\end{proof}

\subsection{Convergence of first passage sets}

In this subsection we prove that the discrete FPS converge to the continuum FPS. Recall that by convention the FPS always contains the boundary of the domain, that $\widetilde D_n$ is the intersection of $D_n$ with $\widetilde \Z_n^2$, and that we use $\widehat \phi_n$ to denote the extension of the metric graph GFF on $\widetilde D_n$ to the rest of $[-C,C]^2$.

\begin{prop}\label{Convergence}
	Suppose $D$ and $(D_n)_{n\in \N}$ are open sets that satisfy condition \hyperlink{logof}{$\logof$}, and $D_n\to D$ in the sense that their complements converge in the Hausdorff topology. Moreover let $\hat \phi_n$ be the extension of the metric graph GFF on $\widetilde D_n$ and suppose that $(u_n)_{n\in \N}$ is a sequence of bounded harmonic functions in $\widetilde D_n$ such that $u_n \rightarrow u$, a bounded harmonic function with piecewise constant boundary values. Denote further for any $z \in D$ by $D^z$ the connected component of $z$ in $D$. Then for any $D^z$, the coupling of the metric graph GFF and its FPS restricted to this component converges in law: $(\widehat \phi_n^{D^z}, (\widetilde \A^{u_n}_{-a} \cap D^z)\cup \partial D^z) \Rightarrow (\Phi^{D^z},\A_{-a}^{u})$ as $n\to \infty$, where $\A_{-a}^u$ is the FPS in the component $D^z$. 
	
	Furthermore, if we couple $(\widehat \phi_n)_{n\in \N}$ and $\Phi^D$ such that $\widehat \phi_n^D\to \Phi^D$ in probability as generalized functions, then $(\widehat \phi_n^{D^z}, (\widetilde \A^{u_n}_{-a} \cap D^z)\cup \partial D^z)  \to (\Phi^{D^z},\A^{u}_{-a})$ in probability.
\end{prop}

\begin{rem}
	The convergence of the open sets $D_n \to D$ in the sense that their complements converge implies, for any $z \in D$ and any $z_n \to z$, the Carathéodory convergence of $(D_n, z_n)$ to $(D^z,z)$. Yet it does not imply that $\partial D_n$ converge to 
	$\partial D^z$ in the Hausdorff metric, hence the need to treat the boundary separately.
\end{rem}

The proof follows from two lemmas. The first one says that the metric graph local sets converge towards continuum local sets. The second one is a general lemma, which in our case will imply that, due to the uniqueness of the FPS, the convergence in law of the pair (GFF, FPS) can be promoted to a convergence in probability. We remark that similar lemmas appear in \cite{SchSh2}, where the authors prove the convergence of DGFF level lines \cite{SchSh2}.

\begin{lemma}[Convergence of metric local sets] \label{Conv ls}
	Suppose $D$ and $(D_n)_{n\in \N}$ are open sets that satisfy condition \hyperlink{logof}{$\logof$}, and $D_n\to D$ in the sense that their complements converge in the Hausdorff topology. Moreover, let $(\tilde\phi_n,A_n)$ be such that $A_n$ is optional for $\tilde \phi_n$ and that for some $c > 0$, the sets $A_n$ have almost surely less than $c$ components none of which reduces to a point\textsl{}. 
	
	Then $(\widehat \phi_n,A_n, (\widehat \phi_n)_{A_n})$ is tight and any sub-sequential limit $(\Phi,A, \Phi_A)$ is a local set coupling. Additionally, for any connected component $D^z$ of $D$ we have that $(\widehat \phi_n^{D^z},(A_n\cap D^z))$ converges to a local set coupling in $D^z$ and $\Phi_{A\cap D^z}$ is given by the restriction of $\Phi_A$ to $D^z$.	
\end{lemma}

\begin{proof}
	Let us first argue tightness. By Lemma \ref{GFFconv} we know that the GFF-s converge in law. Moreover, the space of closed subsets of the closure of a bounded domain is compact for the Hausdorff distance. Hence the sequence $A_n$ is tight. By conditioning on $A_n$, we can uniformly bound the  expected value of the $H^{-1}([-C,C]^2)$ norm of $(\tilde \phi_n)^{A_n}$ and obtain tightness of $(\tilde \phi_n)^{A_n}$ in $H^{-1-\eps}$. Finally, by the Markov decomposition $\tilde \phi-(\tilde \phi_n)^{A_n} = \tilde \phi_{A_n} $ and the triangle inequality, we see that also $(\tilde \phi_n)_{A_n}$ is a tight sequence in $H^{-1-\epsilon}([-C,C]^2)$. Thus, we have tightness of the quadruple $(\tilde \phi_n, A_n, (\tilde \phi_n)_{A_n}, (\tilde \phi_n)^{A_n})$, from which the tightness for  $(\widehat \phi_n, A_n, (\widehat \phi_n)_{A_n}, (\widehat \phi_n)^{A_n})$ also follows.
	
	We pick a subsequence (that we denote the same way) such that $(\tilde \phi_n, A_n, (\tilde \phi_n)_{A_n}, (\tilde \phi_n)^{A_n})$ converges in law to $(\Phi, A, \Phi^1, \Phi^2)$. From the joint convergence we then have that for any bounded continuous functionals $f_{1}$ and $f_{2}$ 
	$$\lim_{n\to +\infty}\E\left[f_{1}((\tilde \phi_n)^{A_n})f_{2}((\tilde \phi_n)_{A_n},A_n)\right] 
	=\E\left[ f_{1}(\Phi^1)f_{2}(\Phi^2,A)\right].$$
	On the other hand, conditionally on $(A_n,(\widetilde \phi_n)_{A_n})$, the law of $(\widetilde \phi_n^{A_n})$ is that of a metric graph GFF in $\widetilde{D}_n \backslash A_n$. By Lemma \ref{GFFconv},  it follows that
	$\E[f_{1}((\tilde\phi_n)^{A_n})\vert A_n,\phi_{A_n}]$ converges a.s. to $ \E[f_{1}(\bar \Phi^A)\mid A]$, where  conditionally on $A$, $\hat \Phi^A$ is a GFF in $D\backslash A$. Thus, by bounded convergence, we have that $\E\left[ f_{1}(\Phi^1)f_{2}(\Phi^2,A)\right]$ is equal to
	\[\lim_{n\to +\infty}\E\left[\E[f_{1}((\tilde \phi_n)^{A_n})\mid A_n, (\tilde \phi_n)_{A_n}]
	f_{2}((\tilde \phi_n)_{A_n},A_n)\right] 
	= 
	\E\left[ \E[f_{1}(\bar \Phi^A)\vert A]f_{2}(\Phi^2,A)\right],\]
	This implies that conditionally on $\Phi^2$ and $A$, the law of $\Phi^1$ is that of a GFF on $D\backslash A$.
	
	Thus, it remains to show that $\Phi^1$ is almost surely harmonic in $D\backslash A$: indeed, then from Lemma \ref{BPLS}, it would follow that $A$ is local and $\Phi^1=\Phi_A$ and $\Phi^2=\Phi^A$.
	
	Let $\Delta_n$ be the discrete Laplacian. From Lemma 2.2 of \cite{ChSm}, it follows that for any smooth function $f$, inside any compact set where derivatives of $f$ remain bounded we have that $\Delta_n f(u)$ is equal to $\Delta f(u)+ O(2^{-n})$. However, from integration by parts it follows that if $f$ is a smooth function with compact support in $D\backslash A$, then $((\tilde \phi_n)_{A_n}, \Delta_n f) =0$ for sufficiently large $n$. Hence $(\Phi^1,\Delta f)=0$ almost surely and thus $\Phi^1$ is harmonic in $D\backslash A$. 
	
	The final claim just follows from Lemma \ref{GFFconv} and the simple fact that if $A$ is a local set for $\Phi$ in a non-connected domain $D$, then for any component of $D$, $D^z$, we have that $A\cap D^z$ is a local set of $\Phi^D$
	
\end{proof}	

The next lemma shows how to promote convergence in law to convergence in probability. See Lemma 4.5 in \cite{SchSh}, and Lemma 31 in \cite{Sha} for earlier appearances in the context of GFF level lines and of Gaussian multiplicative chaos, respectively. We give a slight rewording of the latter proof adapted to our setting.

\begin{lemma} \label{Conv prob}
	Let $(X_n,Y_n)_{n\in \N\cup \{\infty\}}$ be a sequence of random variables in a metric space, living all of them in the same probability space. Suppose we know that
	\begin{enumerate}
		\item $(X_n,Y_n)\Rightarrow (X_\infty,Y_\infty)$
		\item $X_n\to X_\infty$ in probability.
		\item There exists a measurable function $F$ such that $F(X_\infty)=Y_\infty$.
	\end{enumerate}
	Then $(X_n,Y_n)\to(X_\infty,F(X_\infty))$ in probability.
\end{lemma}
\begin{proof}
	Denote $M_n:=(X_n,Y_n,X_\infty,F(X_\infty))$. Because, each coordinate is tight we have that up to a subsequence $M_n\Rightarrow (\bar X_\infty, F(\bar X_\infty), X_\infty, F(X_\infty))$. Thus, any linear combination of them will also converge in law. Note that by (2), $(X_n,X_\infty) \to (X_\infty,X_\infty)$, so $\bar X_\infty = X_\infty$. This fact implies that a.s. $\bar Y_\infty=F(X_\infty)$, thus $Y_n-F(X_\infty)$ converges in law, and therefore in probability, to $0$.
\end{proof}

We have now all the tools to prove the convergence.

\begin{proof}[Proof of Proposition \ref{Convergence}]
	When $\min_{\partial\widetilde D_{n}} u_{n} \geq -a$, we know that $(\tilde \phi_n)_{A_n}+u_{n}$ is constantly equal to $-a$ 
	on $\widetilde D_{n}\backslash A_{n}$ and the claim follows directly from  Lemmas  \ref{Conv ls} and \ref{Conv prob} . 
	
	When $\min_{\partial\widetilde D_{n}} u_{n} < -a$, we can again use the Lemmas \ref{Conv ls} and \ref{Conv prob} to obtain the convergence to a local set $(A, \Phi_A)$ in probability. Moreover, it is easy to see that the conditions (1) and (2) in the Definition \ref{Def ES} hold for $A$, as these properties hold for all approximations and pass to the limit. Thus, it just remains to argue for (3). This condition however follows from Beurling estimate. Pick some component $O$ of the complement of $A$ and any $z$ on its boundary. We can then choose a small enough ball $U^1_z$ around $z$ such that the boundary conditions only change once in this neighborhood. By Beurling estimate (Lemma \ref{Brl}), we can further choose an even smaller ball $U_z$ such that the Brownian motion started inside $U_z \cap O$ exits $O$ through $U^1_z \cap \partial O$ with a probability larger than $1 - \eps/(4\max |u|)$. By the convergence of $A_n \to A$ in probability and Beurling estimate again, we can choose $n_0$ large enough so that for all $n \geq n_{0}$ the metric graph Brownian motion started inside $U_z \cap (\widetilde \Z_n^2 \backslash A_n)$ exits $\widetilde \Z_n^2 \backslash A_n$ through $U^1_z \cap (\widetilde \Z_n^2 \backslash A_n)$ with probability larger than $1 - \eps/(2\max |u|)$ and $u_n - u \leq \eps/2$ uniformly over $\widetilde D_n \cap D$. A final use of Beurling estimate then implies that for any $z_n \in U_z \cap (\widetilde \Z_n^2 \backslash A_n)$, we have that $\tilde h_{A_n}(z_n) + u_n(z_n) \geq \min \{-a, \inf_{w \in U_z \cap \overline O} u(w)\} - \eps$, where $h_{A_n}$ is the metric graph harmonic function outside of $A_n$ as in Proposition \ref{PropStrongMarkov}.The claim follows.
	
\end{proof}

%%%%%%%%%%%%%%%%%%%%%%%%%%%

\subsection{Convergence of clusters of loops and excursions}
\label{SubSecApproxBMClusters}

In this subsection we assume that $u$ is non-negative. Let $\L_{\alpha}^D$ and $\L_\alpha^{\widetilde D_n}$ denote respectively a continuum and metric graph loop-soups of intensity $\alpha \in (0,1/2]$. Similarly, let $\Xi^{D}_{u}$ and 
$\Xi^{\widetilde D_n}_{ u_n}$ denote PPP of boundary-to-boundary excursions in the continuum of intensity $\mu^{D,u}_{\rm exc}$ and in the metric graph setting of intensity $\mu^{\widetilde D_n,u_{n}}_{\rm exc}$ respectively. We sample the loop-soups and PPP of excursions independently and are interested in the clusters of 
$\mathcal{L}_{\alpha}^{D}\cup \Xi^{D}_{u}$ and $\mathcal{L}^{\widetilde{D}_{n}}_{\alpha}\cup \Xi^{\widetilde{D}_{n}}_{u_{n}}$ that contain at least one excursion. By definition two paths belong to the same cluster if they are joined by a finite chain of paths along which two consecutive ones intersect. We denote by $\mathcal{A}=\mathcal{A}(\mathcal{L}^{D}_{\alpha},\Xi^{D}_{u})$ and 
$\widetilde {\AA}_n=\widetilde \AA_{n}(\mathcal{L}^{\widetilde D_n}_{\alpha},\Xi^{\widetilde D_n}_{u_n})$ the closed union of such clusters. 

The main content of this subsection shows that metric graph clusters converge to their continuum counterparts:
\begin{prop}
	\label{PropConvClustExc} Suppose $(\widetilde D_n, z_n)$ satisfy the condition \hyperlink{logof}{$\logof$}  and converge to $(D,z)$ in the Carathéodory sense. Moreover suppose that $u$ is a non-negative bounded harmonic function and $u_{n}\to u$ uniformly on compact subsets of $D$. We also assume that
	whenever $u=0$ on a part of the boundary $\B$, then for any sequence of metric graph boundary points $x_n \to x \in \overline{\B}$ we have that 
	$u_n(x_n) = 0$ as well, for $n$ large enough.
	Then, the sequence of compact sets $(\overline{\widetilde{\AA}_{n}\cap D})_{n\geq 0}$ converges in law for the Hausdorff metric towards $\AA$.
\end{prop}

Let us explain the additional condition on the convergence of $u_n$. 
We want to avoid the following situation. Assume $\B$ is an arc of the boundary
$\partial D$ and $u$ equals 0 on $\B$. Then $\AA$ does not intersect 
$\mathcal{B}$. However one could approximate $u$ by $u_{n}$ small but positive on $\B_{n}\subseteq \partial \widetilde{D}_{n}$ approaching
$\B$. Then almost surely 
$\B_{n}\subset \widetilde {\AA}_n$ and the limit of
$\widetilde {\AA}_n$ would contain $\B$.

Before proving Proposition \ref{PropConvClustExc}, let us show how it allows us to improve the convergence result of for the FPS. Indeed, from Proposition \ref{Convergence} it follows that
$(\widetilde\A^{u_{n}}_{-a}\cap D)\cup\partial D$ converges in law to $\A^{u}_{-a}$. However, by convention $\A^{u}_{-a}$ is defined to contain $\partial D$, and Proposition \ref{Convergence} does not guarantee that there is no part of
$\widetilde\A^{u_{n}}_{-a}$ that for each $n$ intersects $D$ but at the limit converges to a non-trivial arc on $\partial D$. This can be addressed using Proposition \ref{PropConvClustExc}.

\begin{cor}
	\label{CorConvWithoutBoundary}
	Suppose we are in the setting of Proposition \ref{Convergence}. Let 
	$\B$ denote
	\begin{displaymath}
	\mathcal{B}=\lbrace x\in\partial D\vert u(x)\leq -a\rbrace.
	\end{displaymath} 
	Assume that for any sequence of metric graph boundary points 
	$x_n\in\partial\widetilde{D}_{n}$ converging to a point $x \in \overline{\B}$, we have that $u_n(x_n)\leq -a$ for $n$ large enough. Then, 
	the limit of $\overline{(\widetilde\A^{u_{n}}_{-a}\backslash\partial\widetilde{D}_{n})\cap D}$ has empty intersection with the part of the boundary where $u \leq -a$. 
\end{cor}

\begin{proof}

	First assume that $-a\leq\inf u$. 
	Note that $\overline{\A^{u+a}_{0}\backslash\partial D}$ has same law as
	$\AA(\mathcal{L}^{D}_{1/2},\Xi^{D}_{u+a})$. 
	Then, 	$\overline{(\widetilde\A^{u_{n}}_{-a}\backslash\partial\widetilde{D}_{n})\cap D}$ has the law of $
	\overline{(\widetilde\A^{u_{n}+a}_{0}\backslash\partial\widetilde{D}_{n})\cap D}$ that has the law of
	$\overline{\widetilde\AA(\mathcal{L}^{\widetilde D_{n}}_{1/2},
		\Xi^{\widetilde D_{n}}_{u_{n}+a})\cap D}$. Thus, the claim follows from Proposition
	\ref{PropConvClustExc} and the fact that the set $\AA$ does not touch the parts of the boundary with $u = -a$. 
	
	For the general case, consider the boundary condition $u^{\ast} := u\vee (-a)$ and $u_{n}^{\ast} = u_{n}\vee (-a)$ on $D$ and $\widetilde D_n$ respectively. Notice that then $u_n^{\ast}, u^{\ast}$ still satisfy the hypothesis in the statement. Furthermore,
	by monotonicity of the FPS on the metric graph
	$\widetilde\A^{u_{n}}_{-a}\subseteq 
	\widetilde\A^{u^{\ast}_{n}}_{-a}$.
	We conclude by applying the previous case to 
	$\widetilde\A^{u^{\ast}_{n}}_{-a}$.
\end{proof}

Let us now comeback to the proof of Proposition \ref{PropConvClustExc}. The core of our proof is the following lemma, saying that there are no loop-soup clusters that at the same time stay at a positive distance from the boundary, but also come microscopically close to it. 
\begin{lemma}
	\label{LemCrucialConv}
	Let $\alpha\in (0,1/2]$. Suppose that $(\widetilde \Omega_n,w_n)_{n\in \N}$  satisfy \hyperlink{logof}{$\logof$} and $(\widetilde \Omega_n,w_n)\to (\Omega,w)$ in the Carathéodory sense. Then, for all $\delta>0$,
	\begin{displaymath}
	\lim_{\zeta\to 0}\sup_{n\in\mathbb{N}}
	\mathbb{P}\left (\text{There is } \mathcal{C}~\text{cluster of}~
	\mathcal{L}^{\widetilde{\Omega}_{n}}_{\alpha}\text{ s.t. }
	d(\mathcal{C},\partial \widetilde{\Omega}_n)\leq \zeta,
	\sup_{z\in\mathcal{C}}d(z,\partial \Omega)
	\geq\delta\right )=0.
	\end{displaymath}
\end{lemma}
Note that the above lemma is not implied by the convergence result proved by Lupu in \cite{Lupu2015ConvCLE}. However, it could have been proved using the same strategy as in \cite{Lupu2015ConvCLE}. In our article,
we will have a slightly different approach, relying on the convergence of first passage sets.
We will first show how the proposition follows from this lemma, and then prove the lemma.

\begin{proof}
	[Proof of Proposition \ref{PropConvClustExc}] From Lemma \ref{Basic convergences 2} we know that
	$$\lbrace\gamma\in\L_\alpha^{\widetilde D_n}\vert
	\gamma\cap D\neq\emptyset\rbrace\Rightarrow \L_\alpha^{D}, 
	\qquad 
	\lbrace \gamma\in\Xi^{\widetilde D_n}_{u_n}\vert\gamma\cap D\neq\emptyset\rbrace\Rightarrow \Xi^{ D}_{u},$$ 
	as $n\to \infty$. Also $ (\widetilde\AA_n)_{n\in \N}$ is a sequence of random closed sets and thus is tight. Thus, as each coordinate is tight, we can extract a subsequence
	(which we denote in the same way) along which 
	$$(\lbrace\gamma\in\L_\alpha^{\widetilde D_n}\vert
	\gamma\cap D\neq\emptyset\rbrace,
	\lbrace \gamma\in\Xi^{\widetilde D_n}_{u_n}\vert\gamma\cap D\neq\emptyset\rbrace,
	\overline{\widetilde\AA_n\cap D})_{n\in \N}$$
	converges in law to a triple $(\L_\alpha^{D},\Xi^{D}_u,\mathbf A)$. By using Skorokhod's representation theorem, we may assume that this convergence is almost sure. Then, as $\AA$ is a measurable function of $\L_\alpha^D$ and $\Xi_u^D$, it remains to show that $\AA=\mathbf A$ almost surely.
	
	\textit{Let us first show that  $\AA\subseteq \mathbf A$. }
	To do this we consider loops and excursions with cutoff on the diameter and the clusters formed by these loops and excursions. More precisely, respectively in the continuum and on the metric graph, let $\mathcal{A}^{\varepsilon}$ and $\widetilde \AA ^\epsilon$ denote the union of clusters, that are formed of loops and excursions that have diameter greater than or equal to 
	$\varepsilon>0$, and that contain at least one excursion.
	Recall that the diameter is always measured using the Euclidean distance on $\C$, even for paths living on metric graphs.
	
	Note that both $\mathcal{A}^{\varepsilon}$ and 
	$\widetilde\AA^\epsilon$ consist a.s. of finitely many path, and are in particular compact, since a.s. there are finitely many loops and excursions of diameter larger than some value. Now, in our coupling almost surely metric graph loops converge to continuum Brownian loops, metric graph excursions to Brownian excursions, and moreover by (Lemma 2.7 in \cite{Lupu2014LoopsHalfPlane}) their intersection relations also converge. Hence we have that 
	$\overline{\widetilde\AA^\epsilon_n\cap D} \stackrel{a.s.}{\to} \AA^{\epsilon}$. On the other hand $\widetilde\AA_n^\epsilon \subseteq \widetilde\AA_n$ and $\AA^\epsilon \to \AA$ as $\epsilon \to 0$. We conclude that 
	$\AA \subseteq \mathbf A$ almost surely.
	
	\textit{Let us now show that $\mathbf A\subseteq \mathcal{A}$.}
	First notice that there exists a deterministic sequence
	$\epsilon(n)\searrow 0$ such that
	$\overline{\widetilde\AA^{\epsilon(n)}_n\cap D}\stackrel{a.s.}{\to}\AA$. Indeed, as both 
	$\overline{\widetilde\AA^\epsilon_n\cap D}
	\stackrel{a.s.}{\to}\AA^{\epsilon}$ as 
	$n\to \infty$, and $\AA^\epsilon \stackrel{a.s.}{\to}\AA$ as $\epsilon \to 0$ in the Hausdorff distance, we can apply a diagonal argument to choose the sequence $\eps(n)$.
	
	Now, fix a dense sequence of distinct points 
	$(w_i)_{i\geq 0}$ in $D$. Let 
	$\widetilde O_n(w_i)$ and $\widetilde O^{\varepsilon(n)}_n(w_i)$, denote the connected components containing $w_i$ of $\widetilde D_n\backslash\widetilde{\AA}_n$ and $\widetilde D_n\backslash\widetilde{\AA}^{\varepsilon(n)}_n$ respectively. By connected component of $w_{i}$ on a metric graph, we mean the connected component that either contains $w_{i}$ or contains the dyadic square surrounding $w_{i}$. For any fixed $w_i$ it is defined only with certain probability that converges to $1$ as
	$n\to +\infty$. Further, define
	$O(w_i)$ as the connected component of 
	$w_{i}$ in $D\backslash\AA$ and for any $\delta >0$ let $\Theta_{\delta}(w_i)$ be the connected component of $w_i$ in $D\backslash \overline{(\AA+B(0,\delta)})$. As the condition on the boundary convergence of $u_n \rightarrow u$ guarantees that $\mathbf A\cap \partial D =\AA\cap \partial D$, it remains to prove that $\mathbf A\cap D\subseteq \AA\cap D$. To do this it suffices to show that for all $w_{i}$ and $\delta>0$
	\begin{equation} 
	\label{EqCCcomplementary}
	\lim_{n\to +\infty}\mathbb{P}(\Theta_{\delta}(w_i)\subseteq 
	\widetilde O_{n}(w_i))= 1.
	\end{equation}
	For any fixed $w_i$, we will apply Lemma \ref{LemCrucialConv} to 
	$\Omega=O(w_{i})$ and
	$\widetilde{\Omega}_{n}=
	\widetilde{O}_{n}^{\varepsilon(n)}(w_{i})$. Note that 
	$\C\backslash O(w_{i})$ has at most as many connected components as
	$\C\backslash D$. Moreover, from Theorem 1 of \cite{Co} we know that the Hausdorff convergence of 
	$\widetilde\AA^{\varepsilon(n)}_{n}$ to $\AA$ 
	implies the Carathéodory convergence of 
	$(\widetilde{O}_{n}^{\varepsilon(n)}(w_{i}),w_{i})\to
	(O(w_{i}),w_{i})$. Finally, conditioned on 
	$\widetilde\AA^{\varepsilon(n)}_{n}$,
	the law of $\L_\alpha^{\widetilde O_n^{\epsilon(n)}(w_i)}$ (i.e. the law of the metric graph loops of 
	$\mathcal{L}^{\widetilde{D}_{n}}_{1/2}$ that are contained inside
	$\widetilde O_n^{\epsilon(n)} (w_i)$),
	is that of a metric graph loop soup in 
	$\widetilde O_n^{\epsilon(n)} (w_i)$.
	Hence Lemma \ref{LemCrucialConv} implies that
	\begin{equation}
	\label{EqFarClose}
	\lim_{n\to +\infty}\P\left (\text{There is } \mathcal{C}~\text{cluster of}~
	\mathcal{L}_\alpha^{\widetilde{O}_n^{\epsilon(n)}(w_i)}\text{ s.t. }
	d(\mathcal{C},\widetilde{\AA}^{\epsilon(n)}_n(w_i))\leq 2\epsilon(n),
	\sup_{z\in\mathcal{C}}d(z,\partial O(w_i))
	\geq\delta\right )=0.
	\end{equation}
	The metric graph loops that intersect but are not contained in
	$\widetilde\AA_n^{\epsilon(n)}$ are by construction all of diameter smaller than $\epsilon(n)$. Thus, the only way 
	for $\widetilde\AA_{n}$ to have points $\delta$-far from 
	$\widetilde\AA_n^{\epsilon(n)}$ is the event in \eqref{EqFarClose} to be satisfied. We conclude that, with probability converging to 1, we have 
	$\widetilde \AA_n \cap \Theta_\delta(w_i) = \emptyset$. Hence we obtain \eqref{EqCCcomplementary} and conclude the proof of the proposition.
\end{proof}

Now, we present a short proof of the lemma using the already proved convergence of FPS.	The idea is to add Brownian excursions to the loop soup to get an FPS. Then, when the event of having a macroscopic cluster close to the boundary occurs, we use bounds on the FPS and the fact that Brownian excursions intersect any cluster that goes from microscopically close to the boundary to a macroscopic distance, to conclude.

\begin{proof}[Proof of Lemma \ref{LemCrucialConv}]
	
	Notice that by monotonicity of the clusters in $\alpha$, it suffices to prove the claim for $\alpha = 1/2$. By Lemma \ref{Basic convergences}, we can 
	couple 
	$(\L^{\widetilde{\Omega}_{n}}_{1/2})_{n\geq 0}$ and
	$\L^{\Omega}_{1/2}$ in such a way that 
	$\L^{\widetilde{\Omega}_{n}}_{1/2}
	\stackrel{a.s.}{\to}\L^{\Omega}_{1/2}$.
	We also add PPP-s of excursions 
	$\Xi^{\widetilde{\Omega}_{n}}_{n}$
	and $\Xi^{\Omega}_{u}$ for some constant $u>0$ to be chosen later.
	We do it in such a way that $\Xi^{\widetilde{\Omega}_{n}}_{n}$ independent of $\L^{\widetilde{\Omega}_{n}^{2}}_{1/2}$,
	$\Xi^{\Omega}_{u}$ independent of $\L^{\Omega}_{1/2}$, and
	\begin{displaymath}
	\lbrace\gamma\in\Xi^{\widetilde{\Omega}_{n}}_{n}\vert\gamma\cap \Omega\neq\emptyset\rbrace
	\stackrel{a.s.}{\to}
	\Xi^{\Omega}_{u}.
	\end{displaymath}
	
	Now, let us define $$E^{n,\zeta} = \left\{\text{There is } \mathcal{C}~\text{cluster of}~
	\mathcal{L}^{\widetilde{\Omega}_{n}}_{1/2}\text{ s.t. }
	d(\mathcal{C},\partial \widetilde{\Omega}_n)\leq \zeta,
	\sup_{z\in\mathcal{C}}d(z,\partial \Omega)
	\geq\delta\right\}.$$
	Then, by the representation of a metric graph first passage set $(\widetilde \A_0^u)_n$ inside $\widetilde{\Omega}_{n}$ by loops and excursions (Corollary \ref{CorFPCluster}), we can bound $E^{n,\zeta} \subset E_1^{n,u} \cup E_2^{n,\zeta,u}$, where $E_1^{n,u} = 
	\left\{\sup_{z\in (\widetilde \A_0^u)_{n}}d(z,\partial \Omega) \geq \delta/2 \right \}$ and
	\begin{align*}
	E_2^{n,\zeta,u} = \left\{\text{There is } \mathcal{C}~\text{cluster of}~
	\mathcal{L}^{\widetilde{\Omega}_{n}}_{1/2}\text{ s.t. }
	d(\mathcal{C},\partial \widetilde{\Omega}_n)\leq \zeta,
	\sup_{z\in\mathcal{C}}d(z,\partial \Omega)
	\geq\delta \text{, but}~\Xi_{u}^{\widetilde{\Omega}_{n}} \cap \mathcal{C} = \emptyset \right\}.
	\end{align*}
	
	Now, using Proposition \ref{Convergence} for any constant and positive boundary condition $u$, we have that $((\widetilde{\A}^{u}_{0})_{n}\cap \Omega)\cup \partial \Omega \Rightarrow \A_0^u$ in the Hausdorff topology. On the other hand, by convergence of nested local sets (Lemma \ref{BPLS}), monotonicity of FPS (Theorem \ref{Thm::FPS} (3)) and the fact that $\A_0^0 = \partial \Omega$, we know that $\P\left(\sup_{z\in \A_0^u}d(z,\partial \Omega)\geq \delta\right) \rightarrow 0$ as $u \rightarrow 0$. Thus, we get 
	$$\lim_{u\to 0}\limsup_{n\to +\infty}\P(E_1^{n,u})=0.$$
	So, we can chose $u$ such that $\P(E_1^{n,u})$ is arbitrarily small,
	uniformly in $n$ large.
	
	It remains to show that, for any fixed value of $u$, 
	$$\lim_{\zeta\to 0}\limsup_{n\to +\infty}\P(E_2^{n,\zeta,u})=0.$$
	As the excursion measure has infinite mass on the diagonal, it follows that for any fixed $x\in\partial\Omega$, there is a.s. a Brownian excursion in $\Xi^{\Omega}_{u}$ disconnecting 
	$x$ from $\Omega\backslash B(x,\delta/2)$ in $\Omega$. Hence, any connected set joining 
	$x$ to a point at distance $\delta$ from 
	$\partial \Omega$ has to intersect this excursion. However, we know that  $\Xi^{\widetilde{\Omega}_{n}}_{n}$ is independent of $\L^{\widetilde{\mathbb{Z}}_{n}^{2}}_{1/2}$ and that
	$\lbrace\gamma\in\Xi_{u}^{\widetilde{\Omega}_{n}}\vert\gamma\cap\Omega\neq\emptyset\rbrace$ converges in law to
	$\Xi^{\Omega}_{u}$. Thus, the lemma follows.
	
\end{proof}

\section{Consequences of the convergence results}
\label{SecConsec}
In this section, we use Proposition \ref{Convergence} and Proposition \ref{PropConvClustExc} to obtain several results concerning FPS and the Brownian loop soup. These results can be roughly partioned into two: In Section 5.1 we discuss a representation of the FPS with Brownian loops and excursions, and the consequences of this representation: extensions of the isomorphism theorems and several basic properties of the FPS like its local finiteness. In Section \ref{SubSecConseq2}, we discuss consequences on the level lines of the GFF, in particular we prove a convergence result of certain interfaces of the metric graph GFF towards SLE$_4(\rho)$ processes. Let us however start from an easy consequence on the probability of percolation for super-level sets of a metric graph GFF in a large box. This type of percolation questions are for example  studied in \cite{DingLi2018Chemical}.

\begin{cor}[Continuity of percolation]
	\label{CorContPerc}
	Let $\Lambda_{N}$ be the box $\lbrace -N,\dots,N\rbrace^{2}$ in
	$\mathbb{Z}^{2}$ and $\widetilde{\Lambda}_{N}$ the associated metric graph. Let $b>0$ and $\tilde{\phi}_{N}+a$ the metric graph GFF on 
	$\widetilde{\Lambda}_{N}$ with constant boundary condition $a$ on
	$\partial \widetilde{\Lambda}_{N}$. For $\theta\in [0,1]$, we denote 
	$p_{N}(\theta)$ the probability that there is a crossing from	$\partial \widetilde{\Lambda}_{N}$ to 
	$[-\theta N,\theta N]^{2}$
	by positive values of $\tilde{\phi}_{N}+a$.
	Then, the probabilities $p_{N}(\theta)$ are bounded away from $0$ and $1$ also uniformly in $N$, and are moreover continuous in $\theta$ uniformly in $N$.
\end{cor}

\begin{proof} Observe that
	$p_{N}(\theta)$ is the probability that the metric graph FPS
	$(\widetilde{\A}^{a}_{0})_{N}=(\widetilde{\A}^{0}_{-a})_{N}$ 
	intersects $[-\theta N,\theta N]^{2}$. By Corollary 14 in \cite{LupuWerner2016Levy} we see that $p_{N}$ is bounded away from $0$ and $1$ uniformly in $N$.
	
	Let us now consider the continuity in $\theta$. Since for any $\theta$ fixed, a.s. either
	$(\widetilde{\A}^{a}_{0})_{N}\cap [-\theta N,\theta N]^{2}=\emptyset$ or
	$(\widetilde{\A}^{a}_{0})_{N}\cap (-\theta N,\theta N)^{2}\neq\emptyset$, we obtain that
	$p_{N}(\theta)$ is continuous in $\theta$ for any fixed $N$. 
	
	To obtain the uniformity in $N$ we argue as follows.  Let $p(\theta)$ be the probability that the continuum FPS 
	$\A^{a}_{0}$ intersects $[-\theta,\theta]^{2}$. Again, $p(\theta)$ is non-decreasing and continuous, because for fixed $\theta$, a.s. either
	$\A^{a}_{0}\cap [-\theta,\theta]^{2}=\emptyset$ or
	$\A^{a}_{0}\cap (-\theta,\theta)^{2}\neq\emptyset$. But now Proposition \ref{Convergence} tells that $(\widetilde{\A}^{a}_{0})_{N}$
	rescaled by $N^{-1}$ converges in law to $\A^{a}_{0}$ in $[-1,1]^{2}$. Thus, by convergence in law,
	the sequence $(p_{N}(\theta))_{N}$ converges pointwise to $p(\theta)$, and since the functions are non-decreasing, the convergence is uniform in 
	$\theta$. Hence the continuity of $p(\theta)$ gives the uniformity in $N$.
\end{proof}

\begin{rem}
	One can similarly get the continuity in percolation in annuli at macroscopic distance from the boundary of the domain ($\partial \widetilde{\Lambda}_{N}$). For this, the convergence of first passage sets is however not enough. One needs the convergence of all excursion sets, i.e. sign components of $\tilde{\phi}_{N}+b$. This will be done in
	\cite{ALS4}.
\end{rem}

\subsection{Representation of the continuum FPS with Brownian loops and excursions, and consequences on basic properties of FPS}
\label{SubSecConseq1}

From Proposition \ref{CorFPCluster}, we know that a FPS on a metric graph is represented as closure of clusters of metric graph loops and excursions. By using the convergence of the metric graph FPS to the continuum FPS (Proposition \ref{Convergence}) and the convergence of clusters of metric graph loops and excursions to their continuum counterparts (Proposition \ref{PropConvClustExc}), we obtain a similar representation in continuum.
\begin{prop}[FPS $= $ clusters with excursions]
	\label{CorEqFPS}
	Let $u$ be a non-negative harmonic function with piecewise constant boundary values.
	Then, the set 
	$\mathcal{A}(\mathcal{L}^{D}_{1/2},\Xi^{D}_{u})\cup\partial D$ corresponding to the closure of clusters containing excursions, and the first passage set $\A^{u}_{0}$, have the same law.
\end{prop}

We can use this result to obtain a geometric description of the outermost clusters in a Brownian loop-soup $\L^{D}_{1/2}$ when we condition on their outer boundary. More precisely, let $D$ now be simply connected. Then, the outer boundaries of outermost clusters (not surrounded by others) in a Brownian loop-soup $\L^{D}_{1/2}$ are distributed like a conformal loop ensembles CLE$_4$ (\cite{SheffieldWerner2012CLE}). Take one of these boundaries $\Upsilon$ and define $\operatorname{Int}(\Upsilon)$ to be the bounded connected component of $\C\backslash \Upsilon$. It is shown in \cite{QW2015} that conditionally on $\Upsilon$, the Brownian loops in $\operatorname{Int}(\Upsilon)$ that do not touch $\Upsilon$ are distributed like a Brownian loop-soup $\L^{\operatorname{Int}(\Upsilon)}_{1/2}$ inside 
$\operatorname{Int}(\Upsilon)$. Moreover, in the same article the authors prove that conditioned on $\Upsilon$, the loops that intersect $\Upsilon$ are independent from those that do not intersect it, and they have the law of a PPP of Brownian excursions from $\Upsilon$ to 
$\Upsilon$ inside $\operatorname{Int}(\Upsilon)$ with intensity
$\mu_{\rm exc}^{\operatorname{Int}(\Upsilon), 2\lambda}$. Combining this with Proposition \ref{CorEqFPS}, we can give a geometric description of the whole outermost cluster:

\begin{cor}[Cluster of $\L^{D}_{1/2}$ $=$ $\A^{2\lambda}_{0}$=$\A_{-2\lambda}$]
	\label{CorClusterLoopSoup}
	Let the domain $D$ be simply connected. Conditioned on the outer boundary $\Upsilon$ of a Brownian loop-soup cluster in 
	$\L^{D}_{1/2}$, the topological closure of the cluster itself is 
	distributed like a first passage set 
	$\A^{2\lambda}_{0}=\A_{-2\lambda}$ inside 
	$\operatorname{Int}(\Upsilon)$, the interior surrounded by $\Upsilon$.
\end{cor}

One can also combine the isomorphism for the Wick square of the GFF (Proposition \ref{PropIsoURenorm}) and the construction of the FPS from clusters of loops and excursions:

\begin{prop}[FPS $+ $ Wick square]
	\label{PropFPSplusWick}
	Let $u$ be a non-negative harmonic function with piecewise constant boundary values.
	One can couple on the same probability space a GFF 
	$\Phi$ and two point processes 
	$\mathcal{L}^{D}_{1/2}$ and $\Xi^{D}_{u}$ of loops, resp. excursions,
	with $\Xi^{D}_{u}$ independent from $\mathcal{L}^{D}_{1/2}$, such that
	the two following conditions hold simultaneously:
	\begin{enumerate}
		\item $\frac{1}{2}:\Phi^{2}:+ u\Phi+
		\frac{1}{2} u^{2}=L_{\rm ctr}(\mathcal{L}^{D}_{1/2})+
		L(\Xi^{D}_{u})$,
		\item $\A^{u}_{0}=\mathcal{A}(\mathcal{L}^{D}_{1/2},\Xi^{D}_{u})\cup\partial D$.
	\end{enumerate}
\end{prop}

\begin{proof}
	We follow the method of \cite{QW2015} and use subsequential limits of couplings on metric graphs to create a coupling in continuum. As in Propositions \ref{CorFPCluster} and \ref{PropConvClustExc}, we consider metric graph domains 
	$\widetilde{D}_{n}$ converging to $D$ and non-negative bounded metric graph harmonic functions $u_{n}$ converging to $u$. By Propositions 
	\ref{PropIsoCable} and \ref{CorFPCluster} on $\widetilde{D}_{n}$, one can couple a GFF $\tilde{\phi}_{n}$ and loops and excursions 
	$(\mathcal{L}^{\widetilde{D}_{n}}_{1/2},
	\Xi_{u}^{\widetilde{D}_{n}})$ such that
	\begin{enumerate}
		\item $\frac{1}{2}\tilde\phi_{n}^{2}-\E\big[\tilde\phi_{n}^{2}\big]+ 
		u_{n}\tilde\phi_{n}+\frac{1}{2} u_{n}^{2}=
		L(\mathcal{L}^{\widetilde{D}_{n}}_{1/2})
		-\E\big[L(\mathcal{L}^{\widetilde{D}_{n}}_{1/2})\big]+
		L(\Xi^{\widetilde{D}_{n}}_{u_{n}})$,
		\item $\widetilde{\A}_{0}^{u_{n}}=
		\widetilde{\mathcal{A}}_{n}(\mathcal{L}^{\widetilde{D}_{n}}_{1/2},
		\Xi_{u}^{\widetilde{D}_{n}})\cup\partial \widetilde{D}_{n}$.
	\end{enumerate}
	In \cite{QW2015}\footnote{It can be found in the middle of the proof of Lemma 6, starting with the phrase ``\textit{The goal of the following few paragraphs is to explain that the recentered occupation time fields of the cable-system loop-soup can be made to converge to the renormalized occupation time field of $\mathcal L$}''. In fact, their convergence in terms of finite-dimensional marginals can be strenghtened to a convergence, for example, in $H^{-1-\eps}(D)$, but we will not need it here.}, it was shown that 
	$\1_{D}(L(\mathcal{L}^{\widetilde{D}_{n}}_{1/2})
	-\E\big[L(\mathcal{L}^{\widetilde{D}_{n}}_{1/2})\big])$ converges in law to
	$L_{\rm ctr}(\mathcal{L}^{D}_{1/2})$, in the sense that tested against any finite family of smooth functions compactly supported in $D$,
	$f_{1},\dots, f_{k}$, the finite-dimensional vectors converge. 
	Also there one can find the convergence in law of
	$\1_{D}(\tilde\phi_{n}^{2}-\E\big[\tilde\phi_{n}^{2}\big])$ to
	$:\Phi^{2}:$.
	The family of random variables
	\begin{displaymath}
	(\tilde\phi_{n}, 
	\1_{D}(\tilde\phi_{n}^{2}-\E\big[\tilde\phi_{n}^{2}\big]),
	\mathcal{L}^{\widetilde{D}_{n}}_{1/2},
	\Xi_{u}^{\widetilde{D}_{n}},
	\1_{D}(L(\mathcal{L}^{\widetilde{D}_{n}}_{1/2})
	-\E\big[L(\mathcal{L}^{\widetilde{D}_{n}}_{1/2})\big]),
	L(\Xi^{\widetilde{D}_{n}}_{u_{n}}),
	\overline{\widetilde{\mathcal{A}}_{n}\cap D})_{n\in\mathbb{N}}
	\end{displaymath}
	is tight because each component converges in law. Thus, the whole coupling
	has subsequential limits in law, and identities (1) and (2) pass to the limit.
\end{proof}

\subsubsection{Basic properties of the FPS}
\label{BasicFPS}

In this section, we prove several basic but fundamental properties of the continuum FPS: we show that its Hausdorff dimension is a.s. $2$, that it is a.s. locally finite and finally, that it satisfies the FKG inequality.

\begin{cor}[Hausdorff dimension of FPS]
	\label{CorDim}
	Let $u$ be harmonic
	with piecewise constant boundary values. 
	Suppose that
	$\lbrace z\in \partial D\vert u(x)>-a\rbrace$ is non-empty. Then 
	$\A^{u}_{-a}$ has almost surely Hausdorff dimension 2.
\end{cor}

Notice that if $u\leq -a$ then
$\A^{u}_{-a}=\partial D$ almost surely.
\begin{proof}
	First consider the case $u\geq -a$ on $D$.
	Then $\A^{u}_{-a}$ has the law of $\A^{u+a}_{0}$ and by Proposition \ref{CorEqFPS} the first passage set is obtainable from clusters of Brownian loops and excursion. Since the trace of a planar Brownian motion has Hausdorff dimension 2, so has 
	$\A^{u}_{-a}$.
	
	Now we do not assume that $u\geq -a$ everywhere on $D$. Then first sample $\AB_{-a}^{u}$. Then $D\backslash\AB^{u}_{-a}$ has almost surely a connected component $O$ on which $u_{O}>-a$,
	where $u_{O}$ is the harmonic function with boundary condition
	$u$ on $\partial O\cap\partial D$ and $-a$ on 
	$\partial O\cap\partial\AB^{u}_{-a}$.
	Then the first passage set of level $-a$ inside this component $O$, 
	$\A_{-a}^{O,u_{O}}$, is of Hausdorff dimension 2. Since
	$\A_{-a}^{O,u_{O}}\subseteq  \A^{u}_{-a}$, so is $\A^{u}_{-a}$.
\end{proof}

Next, we show that $\A^{u}_{-a}$ is locally finite.

\begin{prop}[Local finiteness of $\A^{u}_{-a}$]
	\label{PropLocFin}
	Let $D$ be a bounded finitely connected domain of $\C$, $u$ a
	harmonic function with piecewise constant boundary values, and $a\in\R$. Then
	$\A^{u}_{-a}$ is locally finite, that is to say that for any 		$\varepsilon>0$, there are finitely many connected components of $D\backslash \A^{u}_{-a}$ of diameter larger than $\epsilon$.
\end{prop}

\begin{proof}
	First, one can assume that $u\geq -a$. If this is not the case, one can first sample $\AB_{-a}^u$ and note that $D\backslash \AB_{-a}^u$ has only finitely many components where $h_{\AB_{-a}}>-a$, and proceed as in the proof of Corollary \ref{CorDim}. For simplicity, we can take $a=0$ and $u\geq 0$.
	
	Let $U$ be an annulus of form
	\begin{displaymath}
	\lbrace z\in\C\vert r<\vert z-z_{0}\vert < 4r\rbrace
	\end{displaymath}
	such that $U\cap D\neq \emptyset$. If $\A^{u}_{-a}$ is not locally finite with positive probability. Then, for some rational $\delta$, at least one annuli with a rational midpoint and with $r = \delta / 4$ is crossed by infinitely many components of $\A^{u}_{-a}$ with positive probability. Thus, it is enough to show that for any fixed annulus $U$ it is a.s. not crossed by an infinity of connected components of $ \A^{u}_{-a}$.
	
	So consider a fixed annulus $U$ and divide it into sub-annuli
	\begin{displaymath}
	U_{\rm int}:=
	\lbrace z\in U\vert \left|  z-z_{0}\right|  <2r\rbrace,
	\qquad
	U_{\rm ext}:=\{ z\in U\vert 3r <\vert z-z_{0}\vert \},
	\qquad 
	U_{\rm mid:}= U \backslash \left( U_{\rm int}\cup U_{\rm ext}\right).
	\end{displaymath}
	We will use the representation of $\A^{u}_{-a}$ by clusters of Brownian loops and excursions as in Proposition \ref{CorEqFPS}. Our aim is to bound the probability of $E(k)$, the event that  there are at least $k$ connected components of $\A^{u}_{-a}$ crossing $U$. To do this, let us first let us consider the following five events:
	\begin{itemize}
		\item $E_{0}(k_{0})$: there are at least
		$k_{0}$ chains of Brownian loops and excursions crossing $U_{\rm mid}$, such that no two different chains contain a common loop or excursion of $\mathcal{L}^{D}_{1/2}\cup\Xi^{D}_{u}$;
		\item $E_{1}(k_{1})$: there are at least $k_{1}$ different Brownian paths (loops or excursions) in 
		$\mathcal{L}^{D}_{1/2}\cup\Xi^{D}_{u}$ that cross $U_{\rm int}$;
		\item $E_{2}(k_{2})$: at least $k_{2}$ different Brownian paths (loops or excursions) in 
		$\mathcal{L}^{D}_{1/2}\cup\Xi^{D}_{u}$ crossing $U_{\rm ext}$;
		\item $E_{3}(k_{3})$: there is at least one loop or excursion in $\mathcal{L}^{D}_{1/2}\cup\Xi^{D}_{u}$ crossing at least 
		$k_{3}$ times $U_{\rm int}$;
		\item $E_{4}(k_{4})$: there is at least one loop or excursion in $\mathcal{L}^{D}_{1/2}\cup\Xi^{D}_{u}$  crossing at least 
		$k_{4}$ times $U_{\rm ext}$.
	\end{itemize}
	We claim that
	\begin{displaymath}
	E(k)\subseteq E_{0}(k_{0})\cup E_{1}(k_{1})\cup E_{2}(k_{2})
	\cup E_{3}(k_{3}) \cup E_{4}(k_{4})
	\end{displaymath}
	whenever
	\begin{equation}
	\label{EqCombinatorics}
	k\geq k_{0}(k_{1}+k_{2})(k_{3}+k_{4}).
	\end{equation}
	Indeed, the $k$ connected components of $D\backslash \A^{u}_{-a}$ crossing $U$ are separated by $k$ chains on Brownian loops and excursions
	that cross $U$. Any two such chains have subchains crossing $U_{\rm mid}$. These subchains may be composed of disjoint paths, or have some Brownian paths in common. In the latter case, the shared paths have to cross either $U_{\rm int}$ or $U_{\rm ext}$ (or both), as two chains cannot be connected inside $U$. Now, suppose that the event
	$E_{1}(k_{1})\cup E_{2}(k_{2})
	\cup E_{3}(k_{3}) \cup E_{4}(k_{4})$ 
	does not hold. Then, any subchain crossing $U_{\rm mid}$ can be connected to at most
	$(k_{1}+k_{2} - 1)(k_{3}+k_{4} - 2)$ others, implying that under \eqref{EqCombinatorics} the event $E(k)$ also cannot hold.
	
	To finish the proof, we just need to argue that $\P(E(k)=\infty)=0$. Let us first note that $\mathbb{P}(E_{1}(k_{1}))$, resp. $\mathbb{P}(E_{2}(k_{2}))$, are tail probabilities of Poisson random variables, and hence decrease faster than $C e^{-k_{i}\log(k_{i})}$, for some constant $C > 0$.	Further, by elementary properties of Brownian paths $\mathbb{P}(E_{3}(k_{3}))$, resp. 
	$\mathbb{P}(E_{4}(k_{4}))$, decrease at least exponentially fast in $k_{3}$, resp. $k_{4}$. Finally, in order to control $\mathbb{P}(E_{0}(k_{0}))$ one can apply the
	van den Berg - Kesten (BK) inequality for Poisson point processes \cite{Berg96BKPoisson}. More precisely, the event $E_{0}(k_{0})$ corresponds to 
	$k_{0}$ disjoint occurrences of the event $E_{0}(1)$ and by BK inequality we have
	\begin{displaymath}
	\mathbb{P}(E_{0}(k_{0}))\leq 
	\mathbb{P}(E_{0}(1))^{k_{0}}.
	\end{displaymath}
	Now, taking for all $i\geq 0$, $k_i=\lfloor k^{\frac{1}{3}}/2\rfloor,$
	so that \eqref{EqCombinatorics} is satisfied, we have that
	\begin{displaymath}
	\P(E(k))\leq C'e^{-C''k^{\frac{1}{3}}},
	\end{displaymath}
	for some constants $C',C''>0$. In particular, this probability tends to $0$ as $k\to +\infty$. 
\end{proof}

Next, we see that $\A^{u}_{0}$ satisfies an Harris-FKG inequality. This follows from the general Harris-FKG inequality for Poisson point processes
(Lemma 2.1 in \cite{Janson1984FKGPPP})

\begin{cor}[Harris-FKG]
	\label{CorFKG}
	Consider a non-negative boundary condition $u$. Let $F_{1}$ and 
	$F_{2}$ be two bounded measurable functionals on compact sets. We assume that $F_{1}$ and $F_{2}$ are increasing, that is to say
	if $K\subseteq K'$, $F_{i}(K)\subseteq F_{i}(K')$. Then
	\begin{displaymath}
	\E[F_{1}(\A^{u}_{0})F_{2}(\A^{u}_{0})]\geq
	\E[F_{1}(\A^{u}_{0})]\E[F_{2}(\A^{u}_{0})].
	\end{displaymath}
\end{cor}

\begin{rem}
	\label{RemFKG}
	One could also obtain a Harris-FKG inequality for $\A^{u}_{-a}$ from a Harris-FKG inequality for the GFF $\Phi$. Then one does not need the constraint $u\geq -a$. First, note that $\A^{u}_{-a}$ is an non-decreasing function of $\Phi$: if $f\in H^{1}_{0}(D)$, $f\geq 0$, then $\A^{u}_{-a}(\Phi)\subseteq\A^{u}_{-a}(\Phi+f)$ a.s. This can be proven similarly to the monotonicity part in Proposition 4.5 in \cite{ALS1}.
	Further, $\Phi$ satisfies itself a Harris-FKG inequality: 
	if $F_{1}$ and $F_{2}$ are functionals such that
	$F_{i}(\Phi+f)\geq F_{i}(\Phi)$ a.s. for $f\in H^{1}_{0}(D)$, 
	$f\geq 0$, then $\E[F_{1}(\Phi)F_{2}(\Phi)]\geq 0$.
	See \cite{Pitt82FKG} for the Harris-FKG property for finite-dimensional Gaussian vectors with covariance matrix having non-negative entries.
\end{rem}

\subsection{Convergence to the level lines and an explicit coupling of level lines}
\label{SubSecConseq2}
In \cite{WernerWu2013FromCLEtoSLE} the authors show that in simply connected domains SLE$_{\kappa}(\rho)$ curves with $\kappa\in(8/3,4]$ can be obtained as ``envelopes'' of clusters of Brownian excursions from boundary to boundary and Brownian loops inside the domain. We will first show how to extend this result to the generalized level lines in multiply connected domains, defined in \cite{ALS1} Section 3, and then use this to prove the main result of this subsection: we show that certain interfaces of metric graph GFFs converge to generalized level lines. In particular, we show that for specific boundary conditions and for simply connected domains, some basic metric graph GFF interfaces converge to SLE$_4(\rho)$ processes.

We work in the following set-up: $D$ is a finitely connected domain and $\partial_{\rm ext} D$ the outermost connected component of $\partial D$. In other words, $\partial_{\rm ext} D$ separates $D$ from infinity. We consider two boundary points
$x_{0}\neq y_{0}\in\partial_{\rm ext}D$ that split 
$\partial_{\rm ext} D$ in two boundary arcs, $\mathcal{B}_{1}$ and $\mathcal{B}_{2}$ (see Figure \ref{Art}). Assume that $u$ is a harmonic function such that on the boundary it is piecewise constant, equal to $-\lambda$ on 
$\mathcal{B}_{2}$, $\inf_{\mathcal{B}_{1}}u>-\lambda$ 
and 
$\inf_{\partial D\backslash\partial_{\rm ext}D}u\geq\lambda$. By Corollary 4.12 of \cite{ALS1}, $\A^u_{-\lambda}$ does not intersect $\B_2$ and thus we can take $O$ the unique connected component of $D\backslash \A^u_{-\lambda}$ such that $\B_2\subseteq \partial O$. Let $\eta$ denote the curve defined by $\A^u_{-\lambda}\cap \partial O$. It is proven in \cite{ALS1} Corollary 4.12 that $\eta$ is a.s. equal to the generalized level line of $\Phi +u$ going from $y_0$ to $x_0$. However, for the rest of this section, one can also treat the generalized level lines as the FPS boundaries just described.

Consider also an independent PPP-s of
loops $\L^{D}_{1/2}$ and boundary-to-boundary excursions
$\Xi^{D}_{u+\lambda}$. By definition there are no excursions hitting 
$\mathcal{B}_{2}\backslash\lbrace x_{0},y_{0}\rbrace$ in $\Xi^{D}_{u+\lambda}$. Let
$\mathscr{D}_{2}$ be the unique connected component of
$D\backslash\AA(\L^{D}_{1/2},\Xi^{D}_{u+\lambda})$ such that
$\mathcal{B}_{2}\subset \partial\mathscr{D}_{2}$ and let
$\partial_{2}\AA(\L^{D}_{1/2},\Xi^{D}_{u+\lambda})=
\partial\mathscr{D}_{2}\cap\AA(\L^{D}_{1/2},\Xi^{D}_{u+\lambda})$. It is also a path in $D$ joining $y_0$ and $x_0$ like the generalized level line $\eta$. The following corollary says that these two paths agree (see Figure \ref{Art} for an illustration):

\begin{figure}[ht!]
	\centering
	\includegraphics[width=3.2in]{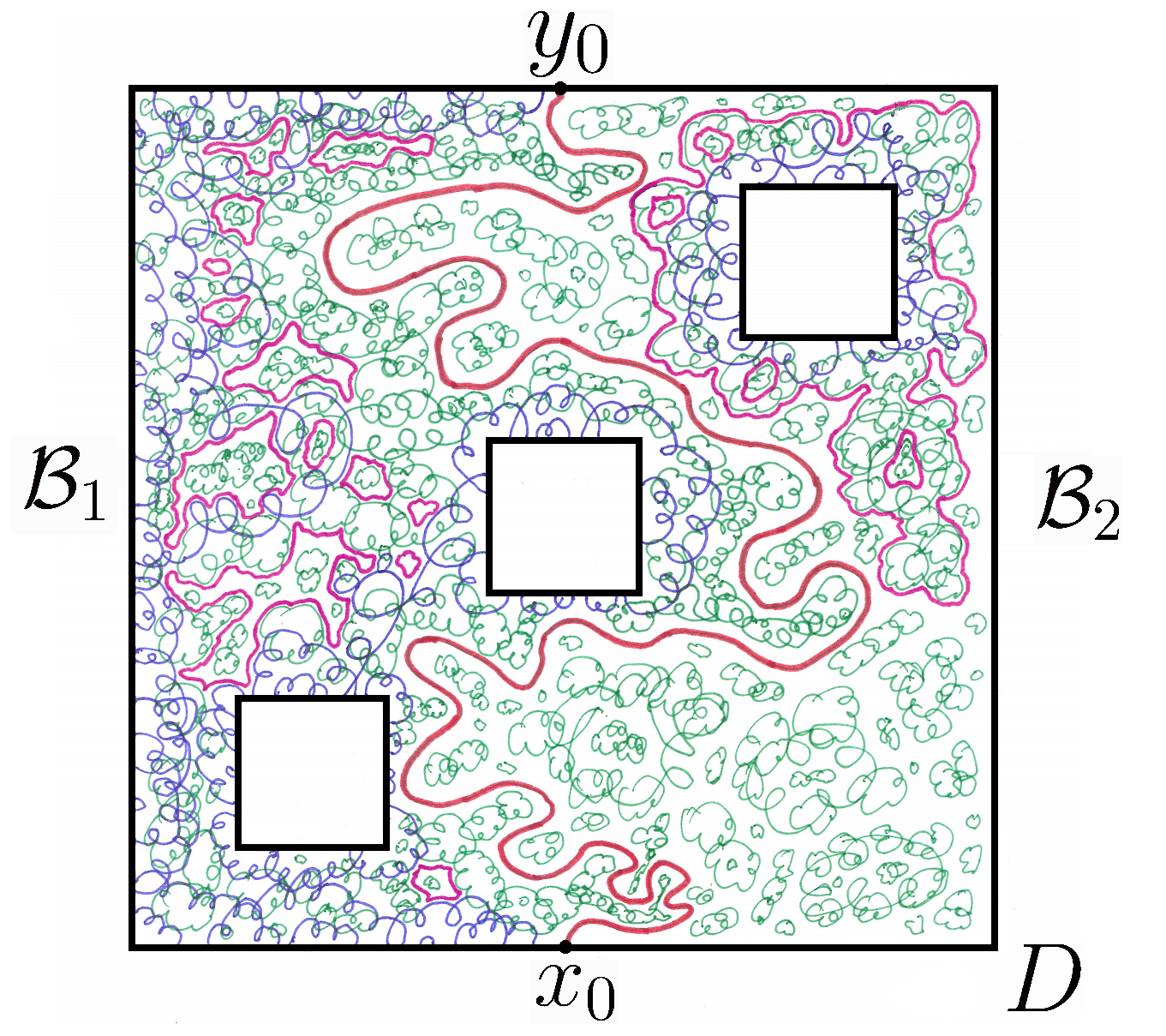}
	\caption{Artistic view of the level line (in red) as envelope of Brownian excursions (in blue) and loops (in green). Magenta contours outline some other boundary components of 
		$\A^{u}_{-\lambda}$.}
	\label{Art}
\end{figure}

\begin{cor}[Level line = envelope of Brownian excursions and loops]
	\label{CorLvlLineCluster}
	Let $D$ be finitely connected and $u$, $\eta$ and 
	$\partial_{2}\AA(\L^{D}_{1/2},\Xi^{D}_{u+\lambda})$ as above. Then the generalized level line $\eta$ has same law as 
	$\partial_{2}\AA(\L^{D}_{1/2},\Xi^{D}_{u+\lambda})$.
\end{cor}

\begin{rem}
	\label{RemLvlLine}
	More generally, other level lines, or families of multiple level lines, can be obtained as boundaries of clusters of Brownian loops and excursions, as long as these level lines are boundaries of a same first passage set. For instance, in a simply connected domain, one can get in this way multiple commuting SLE$_{4}$ curves, which correspond to alternating boundary conditions 
	$0$, $2\lambda$ (Figure \ref{FigMultSLE4}). 
	In \cite{PeltolaWu2017MultComSLE}, the authors give an expression for probabilities of these different pairings.
\end{rem}

\begin{figure}[ht!]
	\centering
	\includegraphics[width=0.4\textwidth]{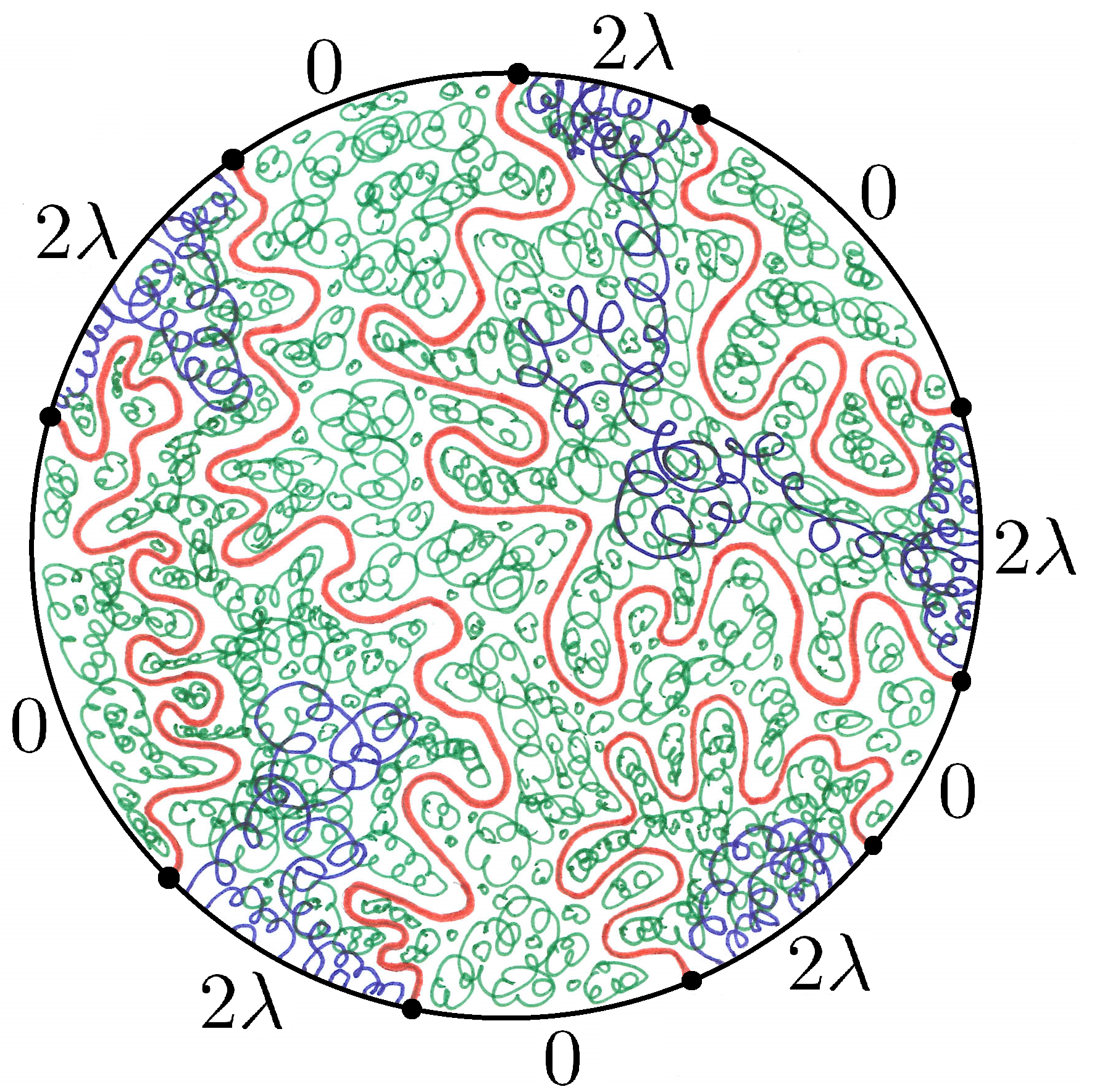}
	\caption{Multiple (here 5) commuting SLE$_{4}$ as boundaries of clusters of Brownian loops (green) and excursions (blue).}
	\label{FigMultSLE4}
\end{figure}

Next, we show that certain interfaces of the metric graph GFF converge in law to level lines of the continuum GFF.
Let $D$, $x_{0}$, $y_{0}$, $\mathcal{B}_{1}$, $\mathcal{B}_{2}$,
$u$, $\eta$ be as previously. Consider 
$\widetilde{D}_{n}$ open subset of $\widetilde\Z^{2}_{n}$ such that we have Hausdorff convergence of $\widetilde{D}_{n}\cup\partial\widetilde{D}_{n}$ to $\overline{D}$ and that of
$\widetilde\Z^{2}_{n}\backslash\widetilde{D}_{n}$ to
$\C\backslash D$. Let $\partial_{\rm ext}\widetilde{D}_{n}$ be the boundary of the only unbounded connected component of
$\widetilde\Z^{2}_{n}\backslash\widetilde{D}_{n}$. We assume that
$\mathcal{B}_{1,n}\cup\mathcal{B}_{2,n}$ is a partition of
$\partial_{\rm ext}\widetilde{D}_{n}$, such that
$\mathcal{B}_{i,n}$ converges to $\mathcal{B}_{i}$, and moreover
$\mathcal{B}_{1,n}$ and $\mathcal{B}_{2,n}$ are separated by exactly two
$2^{-n}\times 2^{-n}$ dyadic squares, of which one contains $x_{0}$ and the other $y_{0}$ (see Figure \ref{FigLvlLM}). Let $u_{n}$ be harmonic on $\widetilde{D}_{n}$ such that $u_{n}$ is constant
$-\lambda$ on $\mathcal{B}_{2,n}$, 
$\inf_{\mathcal{B}_{1,n}}>-\lambda$,
$\inf_{\partial\widetilde{D}_{n}\backslash
	\partial_{\rm ext}\widetilde{D}_{n}}\geq\lambda$ and $u_{n}$ converges to $u$ uniformly on compact subsets of $D$. We have seen that with this boundary conditions, the metric graph first passage set $\widetilde\A^{u_{n}}_{-\lambda}$ contains the boundary $\mathcal{B}_{2,n}$ only by convention, i.e. it satisfies
\begin{displaymath}
\overline{\widetilde\A^{u_{n}}_{-\lambda}\backslash\partial\widetilde{D}_{n}}=\partial\widetilde{D}_{n}\backslash\mathcal{B}_{2,n}.
\end{displaymath}
Let $\partial_{2}\widetilde\A^{u_{n}}_{-\lambda}$ be all the points in $\partial\widetilde\A^{u_{n}}_{-\lambda}$ that are connected in $\A^{u_{n}}_{-\lambda}$ to $\mathcal{B}_{1,n}$ and in $\widetilde{D}_{n}\backslash\A^{u_{n}}_{-\lambda}$ to
$\mathcal{B}_{2,n}$. A.s. 
$\partial_{2}\widetilde\A^{u_{n}}_{-\lambda}$ contains no vertices and the edges it intersects define a path from $x_{0}$ to $y_{0}$ in the dual lattice of $(2^{-n}\Z)^{2}$
(in red on Figure \ref{FigLvlLM}).

\begin{figure}[ht!]
	\centering
	\includegraphics[width=3.4in]{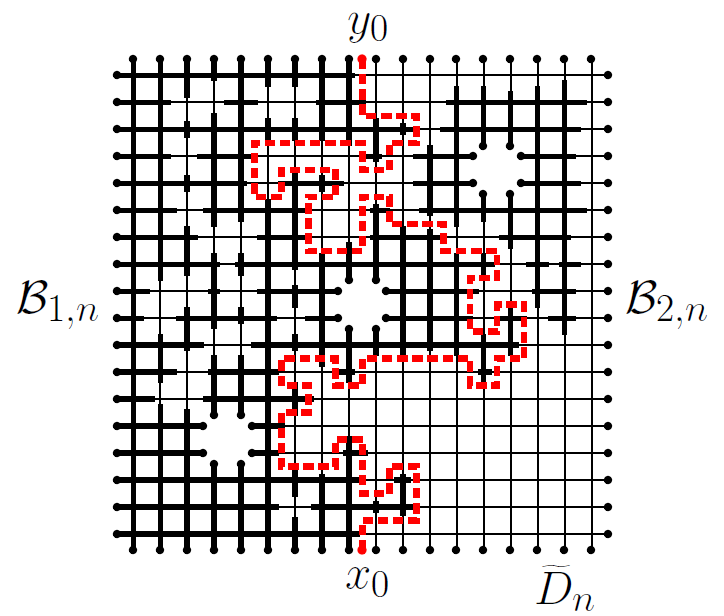}
	\caption{In thick black lines the first passage set 
		$\widetilde{\A}^{u_{n}}_{-\lambda}$ on the metric graph
		$\widetilde{D}_{n}$. Black dots represent 
		$\partial\widetilde{D}_{n}$. The red interface converges in law to a level line of the continuum GFF.}
	\label{FigLvlLM}
\end{figure}

\begin{prop}[Convergence to level lines from metric graph]\label{ConvSLE}
	With the notations above, 
	$\partial_{2}\widetilde\A^{u_{n}}_{-\lambda}$ converges in law for the Hausdorff topology to the level line $\eta$ of the continuum GFF. In particular, if the domain $D$ is simply connected and the boundary condition $u$ is constant
	equal to $b>-\lambda$ on $\mathcal{B}_{1}$,
	$\partial_{2}\widetilde\A^{u_{n}}_{-\lambda}$ converges in law to the trace of an SLE$_{4}(\rho)$ process, with $\rho=b/\lambda -1$.
\end{prop}

Let us stress a corollary that gives a convergence result of very simple metric graph GFF interfaces towards the SLE$_4$:
\begin{cor}\label{convSLE2}
	Consider $\D$ and metric graphs $\widetilde D_n = \widetilde \Z_n \cap \D$. Let the boundary conditions $u, u_n$ be given by $-\lambda$ for all $z$ with $Re(z) < 0$ and by $\lambda$ for all $z$ with $Re(z) \geq 0$. Let $(\widetilde \phi_n)_{n \geq 0}$ be metric graph GFFs on $\widetilde D_n$, and suppose that their extensions converge to a GFF $\Phi$ in probability. Then, the left boundary of the FPS $\widetilde \A^{u_n}_{-\lambda}$ and the right boundary of the FPS $\widetilde \AB^{u_n}_{\lambda}$ both converge in probability w.r.t. the Hausdorff distance to the $-\lambda,\lambda$ level line of \cite{SchSh}, which has the law of SLE$_4$ from $-i$ to $i$ in $\D$.
\end{cor}

\begin{proof}[Proof of Proposition \ref{ConvSLE}]
	$\partial_{2}\widetilde\A^{u_{n}}_{-\lambda}$ is a boundary ``component'' of $\widetilde\A^{u_{n}}_{-\lambda}$ and
	$\eta$ that of $\A^{u}_{-\lambda}$. The convergence of 
	$\widetilde\A^{u_{n}}_{-\lambda}$ to $\A^{u}_{-\lambda}$ in the Hausdorff topology implies that the limit of 
	$\partial_{2}\widetilde\A^{u_{n}}_{-\lambda}$ contains $\eta$ and does not intersect $D_2$, i.e. the 
	$\mathcal{B}_2$ side of $\eta$ (right on Figures \ref{Art} and \ref{FigLvlLM}). Yet this convergence does not exclude that in the limit there are bubbles attached to 
	$\eta$ on its $\mathcal{B}_{1}$ side
	(left on Figures \ref{Art} and \ref{FigLvlLM}). To address this issue, we are going to use the representation of the level line $\eta$ as the boundary of clusters of loops and excursions, and some results from \cite{BrugCamiaLis2014RWLoopsCLE} that state that the clusters of a Brownian loop-soup are ``well connected'', that is to say that, if we remove the microscopic Brownian loops up to some scale, the outer boundaries of clusters do not change too much.
	
	From Corollary \ref{CorLvlLineCluster}, we have the representation 
	$\eta=\partial_{2}\AA(\L^{D}_{1/2},\Xi^{D}_{u+\lambda})$. Consider further metric graph loop-soup
	$\L^{\widetilde{D}_{n}}_{1/2}$, metric graph PPP of excursions
	$\Xi^{\widetilde{D}_{n}}_{u_{n}+\lambda}$ and the union of clusters containing at least one excursion $\widetilde\AA_n = \widetilde\AA_n(\L^{\widetilde{D}_{n}}_{1/2},\Xi^{\widetilde{D}_{n}}_{u_{n}+\lambda})$. Using Lemma \ref{Basic convergences} we can couple everything on the same probability space so that the metric graph PPP and unions of clusters converge to their continuum counterparts.
	
	Now, define
	$\partial_{2}\widetilde\AA_n$ to be the set of points on
	$\partial\widetilde\AA_n$
	that are connected in $\widetilde\AA_n$ to 
	$\mathcal{B}_{1,n}$ and in 
	$\widetilde{D}_{n}\backslash\widetilde\AA_n$ to
	$\mathcal{B}_{2,n}$. As before, it has the same law as 
	$\partial_{2}\widetilde\A^{u_{n}}_{-\lambda}$.
	
	As in the proof of Proposition \ref{PropConvClustExc}, we also consider clusters of loops and excursions that have diameter larger than $\eps$, denoted by $\AA^{\varepsilon} = \AA^{\varepsilon}(\L^{D}_{1/2},\Xi^{D}_{u+\lambda})$ and $\widetilde\AA_n^{\varepsilon}=\widetilde\AA_n^{\varepsilon}
	(\L^{\widetilde{D}_{n}}_{1/2},
	\Xi^{\widetilde{D}_{n}}_{u_{n}+\lambda})$ in the continuum and on the metric graph respectively. Define
	$\partial_{2}\AA^{\varepsilon}$ and
	$\partial_{2}\widetilde\AA_n^{\varepsilon}$ as above. 
	
	From Corollary 5.3 in \cite{BrugCamiaLis2014RWLoopsCLE}, it follows that for fixed $\varepsilon>0$, $\partial_{2}\widetilde\AA_n^{\varepsilon}$ converges as
	$n\to +\infty$ in Hausdorff
	topology to 
	$\partial_{2}\AA^{\varepsilon}$.
	Thus, as $\partial_{2}\AA^{\varepsilon}$ is on the $\B_1$ side of $\eta$, we obtain that
	$\partial_{2}\widetilde\AA_n$
	is asymptotically ``squeezed'' between
	$\partial_{2}\AA^{\varepsilon}$
	and
	$\eta$. 
	
	But now Theorem 4.1 in \cite{BrugCamiaLis2014RWLoopsCLE} implies that as $\varepsilon\to 0$,
	$\partial_{2}\AA^{\varepsilon}$
	converges to
	$\partial_{2}\AA=\eta$ and hence the claim follows.
	
	The result about its law just follows from the fact that level lines in simply-connected domains for piece-wise constant boundary conditions have the law of SLE$_4(\underline{\rho})$ processes \cite{WaWu}.
\end{proof}

\begin{rem}
	Using absolute continuity of level lines, one can extend the convergence result above to the case where the boundary condition is not constantly equal to $-\lambda$ on $\B_{2}$,
	but is less than or equal to $-\lambda$ on $\B'_{2}$ and equal to 
	$-\lambda$ on $\B_{2}\backslash\B'_{2}$, where
	$\B'_{2}\subset\B_{2}$ and $d(\B'_{2},\lbrace x_{0},y_{0}\rbrace)>0$.
\end{rem}

\subsubsection{A coupling with different boundary conditions and coinciding level lines}

Finally, we will discuss how the representation of level lines as boundaries provides an explicit coupling of level lines for the GFF-s with different boundary conditions. Moreover, we also give an exact formula for the conditional probability that the two level lines agree in this coupling, conditioned on one of the level lines. In fact, in the non-boundary touching case, the existance of a coupling where level lines of two GFF-s with different boundary conditions agree with positive probability follows already from Proposition 13 in \cite{ASW}. Here, we provide an explicit such coupling with exact formulas.

Let $D$, $x_{0}$, $y_{0}$, $\mathcal{B}_{1}$, $\mathcal{B}_{2}$,
$u$, $\eta$ be as previously. Moreover let $u^{\ast}$ be another harmonic function that on the boundary is piecewise constant. Suppose $u^{\ast} \geq u$ and let
\begin{displaymath}
\mathcal{B}_{3}=\lbrace x\in\partial D\vert
u^{\ast}(x)>u(x)\rbrace.
\end{displaymath} 
Let $\Phi^{\ast}$ be a GFF. Then we can define $\eta^{\ast}$, a generalized level line of $\Phi^{\ast} + u^{\ast}$
from $y_{0}$ to $x_{0}$. 

\begin{cor}[Coupling of level lines with different boundary conditions]
	\label{CorCoupling}
	Assume $d(\mathcal{B}_{3},\mathcal{B}_{2})>0$.
	Then, there is a coupling of random curves $\eta$ and $\eta^{\ast}$ such that the event $\eta=\eta^{\ast}$ has positive probability. 
	The conditional probability of this event given $\eta$ is
	$$\P(\eta^{\ast}=\eta\vert \eta)=
	\1_{\eta\cap \mathcal{B}_{3}=\emptyset}
	\exp\left(-\mathcal{M}(u,u^{\ast},\eta)\right),$$
	where
	\begin{equation*}
	\begin{split}
	\mathcal{M}(u,u^{\ast},\eta)=&
	\dfrac{1}{2}\sum_{i=1,2}
	\iiint_{\substack{\mathcal{B}_{3}\cap\partial D_{i}
			\\\times\eta\times\mathcal{B}_{3}}}
	(u^{\ast}-u)(x_{1})H_{D_{i}}(dx_{1},dx_{2})
	\mu_{\rm harm}^{D}(x_{2},dx_{3})(u^{\ast}-u)(x_{3})
	\\&\sum_{i=1,2}
	\iiint_{\substack{\mathcal{B}_{3}\cap\partial D_{i}
			\\\times\eta\times
			\partial D\backslash \mathcal{B}_{2}}}
	(u^{\ast}-u)(x_{1})H_{D_{i}}(dx_{1},dx_{2})
	\mu_{\rm harm}^{D}(x_{2},dx_{3})(u+\lambda)(x_{3}),
	\end{split}
	\end{equation*}
	where $H_{D_{i}}(dx_{1},dx_{2})$ is the boundary Poisson kernel on
	$\partial D_{i}\times\partial D_{i}$ and 
	$\mu_{\rm harm}^{D}(x_{2},dx_{3})$ is the harmonic measure on
	$\partial D$ seen from $x_{2}$.
\end{cor}

\begin{rem}
	A crude lower bound for above probability is given by 
	$$\1_{\eta\cap \mathcal{B}_{3}=\emptyset}
	\exp\left(
	-\dfrac{1}{2}\sum_{i=1,2}
	\sup(u^{\ast}-u)^{2}
	M(\mathcal{B}_{3}\cap\partial D_{i},\eta)
	-\sum_{i=1,2}
	\sup(u^{\ast}-u)\sup(u+\lambda)
	M(\mathcal{B}_{3}\cap\partial D_{i},\eta)\right),$$
	where the modulus 
	$M(\mathcal{B}_{3}\cap\partial D_{i},\eta)$ 
	is taken inside $D_{i}$.
\end{rem}

\begin{figure}
	\centering
	\includegraphics[width=3in]{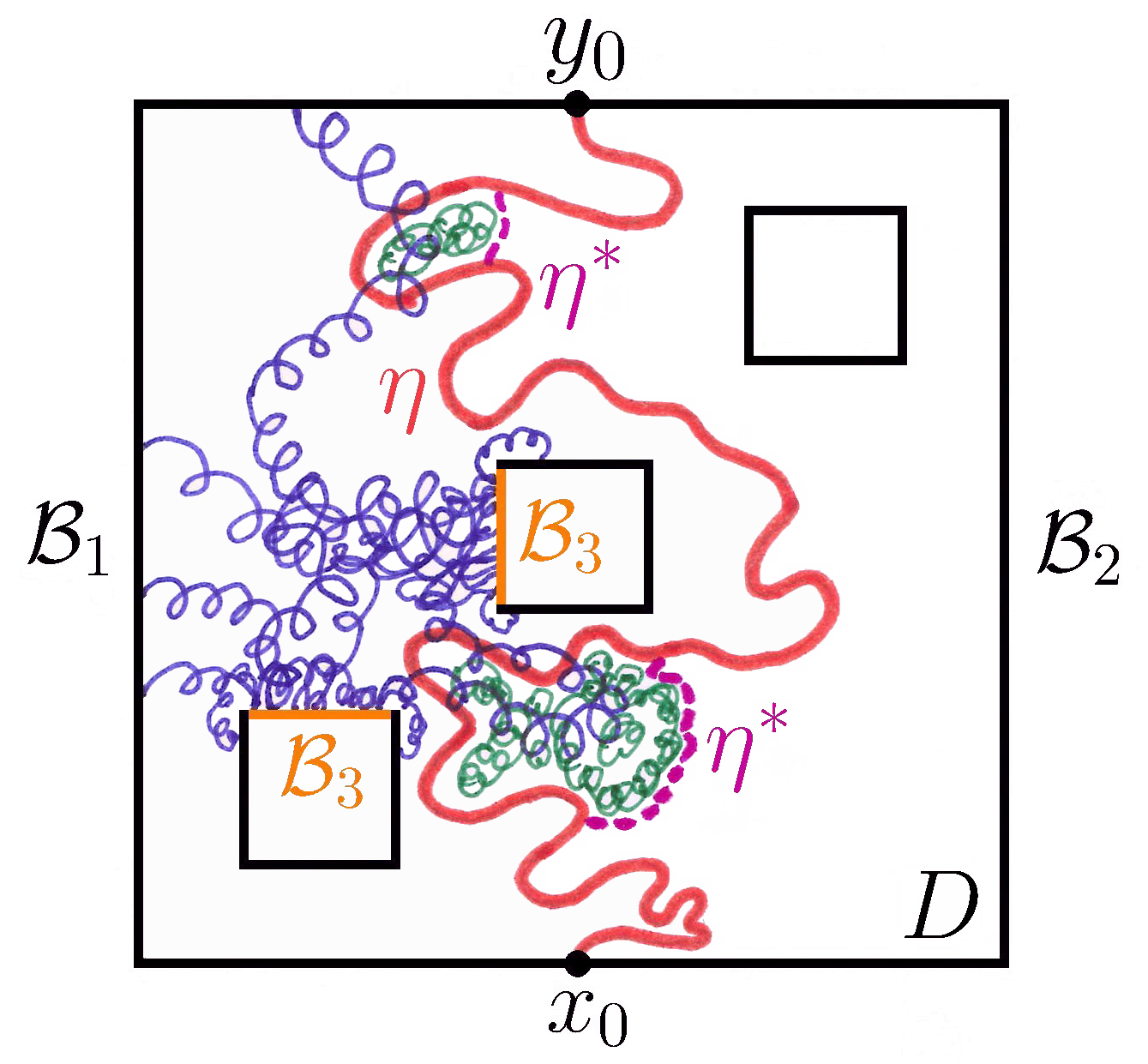}
	\caption{Coupling level lines by adding additional excursions.
		$\B_{3}$ is in orange. In blue are the excursions of
		$\delta\Xi$. Each has at least an endpoint in $\B_{3}$. In green are the clusters of $\L_{1/2}^{D}$ right to $\eta$ that are intersected by
		$\delta\Xi$.}
	\label{FigCouplingLL}
\end{figure}

\begin{proof}
	If we have used the same zero-boundary GFF $\Phi$, then the generalized 0-level lines from $x_{0}$ to $y_{0}$ of
	$\Phi+u$ and $\Phi+u^{\ast}$ would have been a.s. different (unless
	$u^{\ast}=u)$.
	To construct the coupling we rather apply Corollary
	\ref{CorLvlLineCluster} as follows. Consider an independent loop-soup $\L^{D}_{1/2}$, PPP of excursions $\Xi^{D}_{u+\lambda}$ and another PPP of excursions
	with intensity 
	$\mu_{\rm exc}^{D,u^{\ast}+\lambda}-
	\mu_{\rm exc}^{D,u+\lambda}$,
	$\delta\Xi$. Set
	$\Xi^{D}_{u^{\ast}+\lambda}=
	\Xi^{D}_{u+\lambda}\cup\delta\Xi$. 
	We now construct $\eta$ as the envelope of
	$\L^{D}_{1/2}\cup\Xi^{D}_{u+\lambda}$, and
	$\eta^{\ast}$ as the one of
	$\L^{D}_{1/2}\cup\Xi^{D}_{u^{\ast}+\lambda}$
	(Figure \ref{FigCouplingLL}). In this construction,
	\begin{displaymath}
	\P(\eta^{\ast}=\eta\vert \eta)=
	\P(\forall \gamma\in \delta\Xi,\gamma\cap\eta=\emptyset\vert \eta)=
	\1_{\eta\cap \mathcal{B}_{3}=\emptyset}
	\exp\left(-(\mu_{\rm exc}^{D,u^{\ast}+\lambda}
	-\mu_{\rm exc}^{D,u+\lambda})
	(\lbrace\gamma\vert\gamma\cap\eta\neq\emptyset\rbrace)\right),
	\end{displaymath}
	which exactly gives the right expression.
\end{proof}

\section*{Acknowledgement}
The authors wish to thank D. Chelkak for helpful discussions, F. Viklund for useful comments on an earlier version of this paper, W. Werner for many things, and B. Werness for his beautiful simulations and interesting discussions. This work was partially supported by the SNF grants SNF-155922, and SNF-175505. A. Sepúlveda was supported by the ERC grant LiKo 676999. The authors are thankful to the NCCR Swissmap and T. Lupu acknowledges the support of Dr.  Max Rössler,  the Walter Haefner Foundation and the ETH Zurich Foundation. The work of this paper was finished during a wonderful visit of J.Aru and A. Sepúlveda to Paris in May 2018, on the invitation by T. Lupu, funded by PEPS "Jeunes chercheuses et jeunes chercheurs" 2018 of INSMI.

\bibliographystyle{alpha}	
\bibliography{biblio}
\end{document}